\documentclass[a4paper]{article}
\usepackage[english]{babel}
\usepackage{amsmath}
\title{Motivic toposes\\ {\small A research programme}}
\author{Olivia Caramello}
\date{July 22, 2015}

\usepackage[all]{xy}
\usepackage[T1]{fontenc}
\usepackage{tikz,color,makeidx, enumerate}
\usetikzlibrary{matrix,arrows}

\mathcode`\<="4268  
\mathcode`\>="5269  
\mathcode`\.="313A  
\mathchardef\semicolon="603B 
\mathchardef\gt="313E
\mathchardef\lt="313C

\newcommand{\ac}
{`}

\newcommand{\cod}
 {{\rm cod}}

\newcommand{\comp}
 {\circ}

\newcommand{\Cont}
 {{\bf Cont}}

\newcommand{\dom}
 {{\rm dom}}

\newcommand{\empstg}
 {[\,]}

\newcommand{\epi}
 {\twoheadrightarrow}

\newcommand{\hy}
 {\mbox{-}}

\newcommand{\im}
 {{\rm im}}

\newcommand{\imp}
 {\!\Rightarrow\!}

\newcommand{\Ind}[1]
 {{\rm Ind}\hy #1}

\newcommand{\mono}
 {\rightarrowtail}
 
\newcommand{\name}[1]
 {\mbox{$\ulcorner #1 \urcorner$}}

\newcommand{\ob}
 {{\rm ob}}

\newcommand{\op}
 {^{\rm op}}

\newcommand{\Set}
 {{\bf Set }}

\newcommand{\Sh}
 {{\bf Sh}}

\newcommand{\sh}
 {{\bf sh}}

\newcommand{\Sub}
 {{\rm Sub}}

\usepackage{mathrsfs}
\usepackage{amsmath,amsthm}
\usepackage{amssymb}

\newtheorem{theorem}{Theorem}[section]
\theoremstyle{definition}
\newtheorem{definition}[theorem]{Definition}
\newtheorem{lemma}[theorem]{Lemma}
\newtheorem{proposition}[theorem]{Proposition}
\newtheorem{corollary}[theorem]{Corollary}
\newtheorem{remark}[theorem]{Remark}
\newtheorem{remarks}[theorem]{Remarks}

\begin{document}

\maketitle

\begin{abstract}
We present a research programme aimed at constructing classifying toposes of Weil-type cohomology theories and associated categories of motives, and introduce a number of notions and preliminary results already obtained in this direction. In order to analyze the properties of Weil-type cohomology theories and their relations, we propose a framework based on atomic two-valued toposes and homogeneous models. Lastly, we construct a syntactic triangulated category whose dual maps to the derived categories of all the usual cohomology theories. 
\end{abstract}

\tableofcontents

\section{Introduction}

In this work we describe a research programme aimed at constructing `motivic toposes', i.e. classifying toposes of cohomology theories of a certain type (as appropriately formalized within geometric logic). The term `motivic topos' was coined by Luca Barbieri-Viale, who 
observed that there is a formal similarity between the universal property of the classifying topos of a geometric theory and the universal factorization property in the theory of motives and asked whether one could build categories of (mixed) motives by taking the category of abelian groups inside such a topos. 

Broadly speaking, the basic foundational questions underlying the theory of \ac motives' can be summarized as follows:

\begin{itemize}
\item Is it possible to define a \ac universal cohomology theory' from which all the usual Weil-type cohomology theories can be derived?

\item What are the relationships existing between the different cohomology theories?

\item To which extent can one \ac remount' from cohomology to geometry?
\end{itemize}

First-order categorical logic provides a convenient mathematical setting for giving a precise meaning to such questions and looking for possible answers to them. More specifically, the geometry of schemes can be formalized by means of appropriate (many-sorted) languages over which the usual axioms for cohomology theories can be written. The different cohomology functors then become models of the resulting theory. The relationships between them and the geometry of schemes thus acquires a logical interpretation, as the relationships between a theory and its class of (set-based) models. This is a natural subject matter for logicians; for example, the notion of completeness of a first-order theory captures this phenomenon quite precisely, as it means that any two set-based models satisfy exactly the same first-order properties expressible in the language of the theory (we shall come back later to the topic of completeness, approaching it from the perspective of classifying toposes).

The theory of syntactic categories and their completions, including classifying toposes, allows to classify the models of (suitable kinds of) first-order theories by means of universal objects lying in them. Indeed, the syntactic category of a theory $\mathbb T$ contains a universal model $U$ of $\mathbb T$ such that all the other models are obtained, uniquely from it up to isomorphism, as the image of $U$ under a functor preserving the relevant kind of structure. The universal model is built out of the syntax of the theory and satisfies the important property that whatever is valid in it is provable in the theory. It is the place where the syntax and semantics of the theory \ac meet', whence it constitutes a concrete \ac incarnation' of the relationships existing between the different models of the theory. Applying this to theories axiomatizing cohomology theories of a certain kind thus yields universal models of them which can be regarded as incarnations of motives.

Moreover, as remarked in \cite{OC11}, the theory of syntactic categories can be usefully applied to the problem of constructing structures presented by generators and relations. In fact, any syntactic category of a given theory $\mathbb T$ can be regarded as a structure presented by a set of `generators', given by the sorts in the signature of the theory $\mathbb T$, subject to `relations' expressed by the axioms of $\mathbb T$. Conversely, to any structure $\cal C$ one can attach a signature $\Sigma_{\cal C}$, usually called `the' \emph{internal language} of $\cal C$, to express `relations' holding in $\cal C$, consisting of one sort \name{c} for each element $c$ of $\cal C$ and function or relation symbols whose canonical interpretation in $\cal C$ coincide with specified functions or subobjects in $\cal C$ in terms of which the designated `relations' holding in $\cal C$ can be formally expressed; over such a canonical signature one can then write down axioms so to obtain a theory whose syntactic category ${\cal C}^{\cal S}_{\mathbb T}$ satisfies the universal property of `the free structure on $\cal C$ subject to the relations $R$'. 

Classifying toposes are particularly relevant in this context since they satisfy similar universal properties but, unlike the latter, they possess all small limits, colimits and exponentials and hence can be used for performing constructions which \emph{a priori} are not possible in smaller syntactic categories, as well as for investigating whether a given syntactic category is closed under (any of) these constructions. 

Also, allowing models of theories in arbitrary toposes, rather than just in the classical category of sets, gives the possibility in particular of considering \ac (co)homology theories in a topos', in the sense of models in that topos of a theory axiomatizing (co)homology theories. It can be the case that cohomology theories with particular properties exist in specific toposes whilst there are no such theories in the classical set-based context. For instance, it is well-known that a cohomology theory with coefficients in $\mathbb Q$ which \ac generates' all the usual cohomology theories by change of base does not exist (at least in characteristic $p$), but if we relax the requirement that the cohomology theory should be a set-based one to allow cohomology theories in toposes, we actually get the existence of such a theory (cf. section \ref{sec:motivictheories}).     

Moreover, as argued in \cite{OC10}, classifying toposes provide a natural unifying setting for comparing geometric theories which are defined differently but which have the same, or a strictly related, mathematical content. In fact, as we argue in section \ref{rem_independence}, $\ell$-adic and $\ell'$-adic cohomologies, for $\ell,\ell'$ different from the characteristic of the base field, should give rise to different representations of a unique \ac motivic topos'.    

For all these reasons, syntactic categories and classifying toposes seem to us particularly relevant in connection with the problem of building categories of motives.

In \cite{NoriSyntactic}, we began applying these ideas in the context of Nori motives. We showed that Nori's category of a representation $T$ of a diagram $D$ with values in $k$-vector spaces can be realized as the effectivization of the regular syntactic category of the theory $\textrm{Th}(T)$ consisting of all the regular sequents over the language $L_{D}$ of $D$ which are satisfied by $T$. Speaking in terms of generators and relations, we can describe Nori's category by saying that its generators are the objects and arrows of the diagram $D$ and the relations are the regular sequents satisfied by the representation $T:D \to k\textrm{-vect}$.

In the same paper we also gave necessary and sufficient explicit conditions for two representations $T$ and $T'$ of the same diagram to yield equivalent categories of Nori motives, interpreting them as the requirement that $\textrm{Th}(T)=\textrm{Th}(T')$. If $D$ is Nori's diagram (or some enlargement of it given for instance by our syntactic triangulated category introduced in section \ref{sec:triangcat} below) and $\cal K$ is a class of cohomology theories $T$ such that for any $T$ and $T'$ in $\cal C$, $\textrm{Th}(T)=\textrm{Th}(T')$ then the classifying topos ${\cal E}_{T}$ of the theory $\textrm{Th}(T)$ has as points all the theories in $\cal K$ as well as those which satisfy the same regular sequents over $L_{D}$ as them. The realisation functors (i.e., the inverse image functors of the point of it), restricted to the subcategory of ${\cal E}_{T}$ given by the effectivization of the regular syntactic category of $\textrm{Th}(T)$, are exact and faithful whence ${\cal E}_{T}$ is a good candidate for a motivic topos. Still, ${\cal E}_{T}$ depends on $T$ and has as points all the theories in $\cal C$ only \emph{conditionally} to the fact that $\textrm{Th}(T)=\textrm{Th}(T')$ for all $T, T'$ in $\cal K$. 

It is worth to note that the bounds on dimension of homology groups are not expressible in regular logic over $L_{D}$, whence the abelian $\mathbb Q$-linear category of Nori-type motives defined in \cite{NoriSyntactic} cannot distinguish between two cohomology theories which satisfy the same regular properties but whose groups have different dimensions. So, for instance, we could imagine that $\ell$-adic cohomology and higher $K$-theory give rise to the same category of Nori-type motives, in spite of having different dimensions in some degree. It would therefore be desirable to have a direct construction of a \ac motivic topos' which is independent from any particular realization $T$ and which is defined as the classifying topos of a geometric theory expressing the fundamental features of cohomology theories belonging to a certain class. We propose in the paper a strategy for constructing such a topos, based on the topos-theoretic interpretation of Fra\"iss\'e's construction established in \cite{OC2}. More specifically, we identify the framework of atomic two-valued toposes as a particularly suitable one for building motivic toposes; as shown in \cite{OC2}, these toposes are precisely the classifying toposes of atomic complete theories or, equivalently, of theories of homogeneous $\mathbb S$-models for a theory of presheaf type $\mathbb S$. 

To test our research strategy, we define in section \ref{sec:basictheorypresheaftype} a theory of presheaf type axiomatizing substructures of (co)homology theories of a certain kind (this theory is not supposed to be the definitive one, but a good \ac approximation' to it), and investigate in section \ref{sec:theoryhomogeneousmodels} the theory of its homogeneous models; if the usual Weil-type homology theories are homogeneous in this sense then the classifying topos of this theory will satisfy all the properties naturally expected of a \ac motivic topos'. We notably show that the exactness conditions for (co)homology theories \emph{follow} from homogeneity. In fact, homogeneity can be regarded as a strengthening of exactness for distinguished pairs of definable arrows over the signature of the theory; it is therefore not unreasonable to expect that all the usual (co)homology theories satisfy it.

In the final section of the paper we construct a syntactic triangulated category which maps into all the derived categories of sheaves associated to the usual (co)homology theories. This category extends Nori's quiver and satisfies the property, important for our purposes, that all the images of equationally definable arrows in it are equationally definable.  
 
This paper is partly written in an informal, programmatic style privileging intuitions over formal statements. In fact, since this research programme is a long-term one, we have chosen to share it with the community at this stage rather than waiting until the moment where the first concrete results come out.

\section{General considerations}

In this section we give a general description of our  programme for building \ac motivic toposes'. 

Let us start by recalling the fundamental notion of classifying topos, which plays a central role in our approach.

\subsection{Classifying toposes}\label{sec:classifyingtoposes}

The classifying topos of a geometric theory $\mathbb T$ is by definition, a Grothendieck topos ${\cal E}_{\mathbb T}$ satisfying the following universal property:
\[
\textbf{Geom}({\cal E}, {\cal E}_{\mathbb T})\simeq {\mathbb T}\textrm{-mod}({\cal E}),
\]
for any Grothendieck topos $\cal E$, naturally in $\cal E$, where $\textbf{Geom}({\cal E}, {\cal E}_{\mathbb T})$ is the category of geometric morphisms ${\cal E} \to {\cal E}_{\mathbb T}$ and ${\mathbb T}\textrm{-mod}({\cal E})$ is the category of $\mathbb T$-models in $\cal E$. In particular, taking $\cal E$ is equal to $\Set$, one gets that the points of ${\cal E}_{\mathbb T}$ correspond precisely to the classical (set-based) models of $\mathbb T$. 

By a theorem of Joyal-Makkai-Reyes, every \emph{geometric theory} (i.e., a first-order theory whose axioms can be presented ) has a classifying topos, which can be built as the category of sheaves $\Sh({\cal C}_{\mathbb T}, J_{\mathbb T})$ over the syntactic site $({\cal C}_{\mathbb T}, J_{\mathbb T})$. Recall that the category ${\cal C}_{\mathbb T}$ has as objects the geometric formulae-in-context $\{\vec{x}. \phi\}$ over the signature of $\mathbb T$ and as arrows $\{\vec{x}. \phi\}\to \{\vec{y}. \psi\}$ the $\mathbb T$-provable equivalence classes $[\theta]$ of $\mathbb T$-provably functional formulae $\theta(\vec{x}, \vec{y})$ from $\{\vec{x}. \phi\}$ to $\{\vec{y}. \psi\}$, while $J_{\mathbb T}$ is the canonical topology on ${\cal C}_{\mathbb T}$ (a sieve $\{[\theta_{i}]:\{ \vec{x_{i}}. \phi_{i}\}\to \{\vec{y}. \psi\} \textrm{ | } i\in I\}$ in ${\cal C}_{\mathbb T}$ is $J_{\mathbb T}$-covering if and only the sequent $(\psi \vdash_{\vec{y}} \mathbin{\mathop{\textrm{ $\bigvee$}}\limits_{i\in I}} (\exists \vec{x_{i}})\theta_{i}(\vec{x_{i}}, \vec{y}))$ is provable in $\mathbb T$). 

\subsection{Motivic toposes}

Applying the existence theorem for classifying toposes in the case of a geometric theory $\mathbb T$ formalizing cohomology theories $T$ of a certain type, one obtains a topos ${\cal E}_{T}$ containing a universal model $U$ of $\mathbb T$ such that every $T$ can be obtained from $U$ by applying a \ac realization functor' to it, namely the inverse image of the point of ${\cal E}_{T}$ given by $T$. Recall that inverse images of geometric morphisms of toposes preserve all finite limits and all small colimits. They are not faithful in general (geometric morphisms whose inverse image functor is faithful, equivalently conservative, are called \emph{surjections}). In fact, from a logical point of view, the faithfulness of the inverse image of the geometric morphism $f_{M}:{\cal E}\to {\cal E}_{\mathbb T}$ corresponding to a $\mathbb T$-model $M$ in $\cal E$ amounts precisely to the requirement that the model $M$ be \emph{conservative}, in the sense that every geometric sequent over the signature of $\mathbb T$ which is valid in $M$ is provable in $\mathbb T$.   

The basic requirement for a category $\cal M$ of (mixed) motives is that $\cal M$ should be a rigid $\mathbb Q$-linear Tannakian category with a representation $i$ of Nori's diagram (or of any other larger diagram formalizing schemes) into it with the property that there should be exact and faithful realisation functors $f_{T}:{\cal M} \to k\textrm{-vect}$ for each known Weil cohomology theory $T:D \to k\textrm{-vect}$ such that $f_{T}\circ i=T$.
 
The question thus naturally poses as to how to construct a geometric theory $\mathbb T$ formalizing (co)homology theories of a certain kind in such a way that an appropriate subcategory $\cal S$ of the category of internal $\mathbb Q$-vector spaces in its classifying topos satisfies the above-mentioned requirements for a category of (mixed) motives.

First of all, we notice that the signature $\Sigma$ of $\mathbb T$ should be an expansion of the signature having one sort $\name{X}$ for each scheme $X$ and the constants and function symbols on each $\name{X}$ formalizing the abelian group structure on them, so that every (co)homology theory can be regarded as a $\Sigma$-structure (in the sense of classical model theory). To formalize morphisms of homology or cohomology groups, one should add function symbols of the appropriate variance. For instance, $\Sigma$ can be taken equal to the language $L_{D}$ introduced in \cite{NoriSyntactic}, where $D$ is Nori's diagram, or it can be taken equal to the internal language of the syntactic triangulated category defined in section \ref{sec:triangcat}. One could also include in $\Sigma$ function symbols corresponding to scheme correspondences, since, as it is well-known, they induce morphisms at the level of (co)homology groups. 

We should mention that a signature of this kind for formalizing cohomology theories was already considered by A. Macintyre in \cite{Macintyre}.

There are two different possibilities for axiomatizing the vector space structures on the (co)homology groups: either one fixes the coefficients field (in which case, if one wants to talk about all the usual (co)homology theories, one is forced to take this field equal to $\mathbb Q$) or one leaves it variable, so to be able for instance to compare, using logical means, the dimensions of the groups of different (co)homology theories over their respective coefficients fields. For an approach implementing the first possibility, see section \ref{sec:bottomup}, while for an approach implementing the second, see section \ref{sec:topdown}. 

Given the fact that ${\cal E}_{\mathbb T}$, as any Grothendieck topos, possesses finite limits, coequalizers of equivalence relations, internal tensor products (of internal vector spaces over a given internal field) as well as internal duals of internal vector spaces over an internal field, it is natural to require $\cal S$ to be a full subcategory of ${\cal E}_{\mathbb T}$ which is closed with respect to finite limits and coequalizers of equivalence relations - so that the exactness requirement for the realisation functors follows at once from the universal property of the classifying topos - as well as with respect to internal duals and internal tensor products. Recall that a vector space $V$ over a field $K$ is said to be \emph{reflective} if the canonical map $V \to V^{\ast \ast}$ is an isomorphism. It is well-known that the category of reflective vector spaces over a field $K$ is a rigid abelian $K$-linear tensor category. This proof can be internalized to an arbitrary Grothendieck topos, yielding the following result: for any internal field $K$ in a Grothendieck topos $\cal E$, the category of reflective internal $K$-vector spaces in $\cal E$ (meaning the $K$-vector spaces $V$ in $\cal E$ such that the canonical arrow $V\to V^{\ast \ast}$ is an isomorphism, where the dual of $V$ is defined by taking the internal hom $V^{\ast}=\textrm{Hom}_{K}(V, K)$) is a rigid abelian tensor category. 

The closedness of $\cal S$ with respect to finite limits and coequalizers of equivalence relations can be obtained for instance by defining $\cal S$ equal to the abelian subcategory of the category of internal $\mathbb Q$-vector spaces generated by the objects and arrows in ${\cal E}_{\mathbb T}$ which interpret the sorts and function symbols of the signature of $\mathbb T$, as it was done for instance in \cite{NoriSyntactic}.

Given the fact that ${\cal E}_{\mathbb T}$ is the $\infty$-pretopos completion of the geometric syntactic category ${\cal C}_{\mathbb T}$ of $\mathbb T$, it is natural to seek a characterization of $\cal S$ as a (completion of a) syntactic category of $\mathbb T$ in a suitable fragment of logic to which $\mathbb T$ belongs (for instance, in \cite{NoriSyntactic} $\cal S$ was characterized as the effectivization of the regular syntactic category of the regular theory $\mathbb T$).

The faithfulness requirement for the realisation functors is a subtle matter. First of all, the faithfulness of the restriction to $\cal S$ of a realisation functor $f^{\ast}$ does \emph{not} imply the faithfulness of $f^{\ast}$ on the whole ${\cal E}_{\mathbb T}$ (see for instance section \ref{sec:bottomup}); on the other hand, the converse is obviously true, and one could look for toposes whose points are all surjections, such as atomic two-valued toposes (see section \ref{sec:topdown}). 

From a logical point of view, the faithfulness requirement for realisation functors means precisely that all the different (co)homology theories have the same \ac logical content': more precisely, if $\cal S$ is (a completion of the) syntactic category of $\mathbb T$ within a given fragment of logic $L$ to which $\mathbb T$ belongs, it means that all the (co)homology theories satisfy exactly the same sequents written in the fragment $L$ (see for instance section 2.4 of \cite{NoriSyntactic}). This brings us into the related subject of \emph{completeness} for first-order theories. Recall that a theory $\mathbb T$ within a given fragment $L$ of first-order logic is $L$-complete if every assertion (i.e. closed formula) in $L$ over its signature is either provably true or provably false in $\mathbb T$. It is worth to note that completeness varies depending on the fragment of logic that one considers; that is, the fact that a given theory, regarded as belonging to a given fragment of logic $L$, is $L$-complete does not imply that it should be $L'$-complete for any larger fragment $L'$ (even though the converse always holds). 

As shown in \cite{OC2} and \cite{OC13}, completeness of a geometric theory (within geometric logic) amounts precisely to the \emph{two-valuedness} of its classifying topos (recall that a Grothendieck topos is two-valued if the only subobjects of its terminal object are the zero subobject and the identity subobject, and they are distinct from each other), while first-order completeness of a finitary first-order theory amounts precisely to geometric completeness of its Morleyization. On the other hand, as shown in \cite{OC5}, the geometric completeness of an atomic theory implies its completeness also with respect to finitary first-order (not necessarily geometric) assertions over its signature. 

The importance of the topos-theoretic characterization of completeness as two-valuedness lies in the fact that it paves the way, in light of the technique \ac toposes as bridges' of \cite{OC10}, for investigating the logical matter of completeness of theories using techniques of other mathematical domains providing alternative representations of the given classifying topos. For instance, it was shown in \cite{OC2} that the theory of homogeneous $\mathbb T$-models, for a theory $\mathbb T$ of presheaf type whose category $\textrm{f.p.} {\mathbb T}\textrm{-mod}(\Set)$ of finitely presentable $\mathbb T$-models satisfies the amalgamation property, is complete if and only if the category $\textrm{f.p.} {\mathbb T}\textrm{-mod}(\Set)$ satisfies the joint embedding property. Needless to say, the latter property is much easier to verify in practice than the completeness condition (see \cite{OC2} for a few examples). 

We can now describe our two main approaches to the problem of constructing \ac motivic toposes'.

\subsection{The \ac bottom-up' approach}\label{sec:bottomup}

The \ac bottom-up' approach consists in building \ac motivic toposes' starting from a specific (co)homology theory $T$, and in trying to prove that (co)homology theories belonging to a given class for which one wants to construct motives yield equivalent toposes (or equivalent categories of motives).

This is the approach that we followed in \cite{NoriSyntactic}, where we realized Nori's category of a representation $T:D\to k\textrm{-vect}_{\textrm{fin}}$ as the full subcategory of classifying topos of the theory $\textrm{Th}(T)$ on the supercoherent objects. Unlike Nori's construction, our syntactic construction of a category ${\cal C}_{T}$ satisfying the universal factorization property makes sense also for representations with values in infinite-dimensional vector spaces over a given field $k$. So, for any known (co)homology theory, regarded as a representation $T:D\to {\mathbb Q}\textrm{-vect}$, one can consider the classifying topos ${\cal E}_{T}$ of the theory $\textrm{Th}(T)$. This will be a \ac motivic topos' if for any other known homology theory $T'$, there is an equivalence ${\cal E}_{T}\simeq {\cal E}_{T'}$ commuting with the canonical representations $D\to {\cal E}_{T}$ and $D\to {\cal E}_{T'}$, that is if and only if $\textrm{Th}(T)=\textrm{Th}(T')$, in other words if and only if $T$ and $T'$ satisfy the same regular properties over the language $L_{D}$. We gave a more concrete reformulation of these conditions in Corollary 2.10 \cite{NoriSyntactic}. It seems likely that all the usual homology theories satisfy these properties, which - it is important to remark - do \emph{not} detect the possible differences concerning the dimensions of homology groups pertaining to distinct theories. 

Notice that these toposes ${\cal E}_{T}$ are all connected and locally connected (as they are the classifying toposes of regular theories), but not necessarily atomic nor two-valued.

Provided that there exists a kind of \ac product' in $D$ compatible with $T$ and a kind of \ac involution' in $D$ compatible with $T$, it should be possible to show that the category ${\cal C}_{T}$ is closed in ${\cal E}_{T}$ under tensor products over $k$ and duals. We plan to investigate these points in due course. 

Given the fact that the dimension of a (co)homology group can be read straight off from the rigid tensorial structure (as the scalar arising by considering the composition of the two canonical maps $1 \to V \otimes V^{\ast} \to 1$, where $1$ is the unit of the tensor structure), it is natural to expect that the differences in the dimensions of the groups (subsisting for example between $\ell$-adic cohomology and higher $K$-theory) will reflect into the different behaviour, in terms of preservation or non-preservation of the rigid tensorial structure, by the relevant realisation functors.  

Summarizing, the \ac bottom-up approach' consists in constructing toposes which are \ac motivic' conditionally to the fact that the different (co)homology theories satisfy the same properties (over a fixed language and within a given logic - in the case of \cite{NoriSyntactic} these were the language $L_{D}$ and regular logic). The resulting realisation functors defined on the relevant subcategories of motives will all be faithful and exact but, depending on the (co)homology theory in question, they might not preserve the rigid tensorial structure. For theories whose associated realisation functors preserve such a structure, the equality of the dimensions of the respective groups follows as a straightforward consequence of this preservation property. 

We shall now describe an alternative, more axiomatic, approach to the problem of constructing \ac motivic toposes'.

\subsection{The \ac top-down' approach}\label{sec:topdown}

Given a class of cohomology theories $\cal K$ for which one wants to construct \ac motives', we can attempt to axiomatize the properties which are common to all the theories in $\cal K$ by means of a geometric theory ${\mathbb T}_{\cal K}$ whose classifying topos has enough points. If all the theories in $\cal K$ carry in themselves the same amount of \ac information' then the realisation functors defined on the classifying topos of ${\mathbb T}_{\cal K}$ will be (exact and) faithful (as they will correspond to conservative models of ${\mathbb T}_{\cal K}$). This topos will satisfy the requirements of a \ac motivic topos'; indeed, it will be an incarnation of the properties (expressible within geometric logic in the language of ${\mathbb T}_{\cal K}$) that the theories in $\cal K$ share with each other.

The signature $\Sigma_{\cal K}$ of ${\mathbb T}_{\cal K}$ will contain one sort for each of the (co)homology groups that one wants to study and function symbols corresponding to the morphisms between these groups that one wishes to consider, plus possibly other symbols formalizing the structures that one wants to consider on them (e.g., the field structure on the coefficient field $H_{0}(\textrm{Spec}(k), \emptyset)$ of the theory and the vector space structure over this field on all the other homology groups). Since (co)homology groups are indexed by tuples consisting of schemes over a given base field and integers (for instance triplets $(X, Y, i)$ index relative homology groups $H_{i}(X, Y)$), one should take a sort for each indexing object of this kind.

This approach allows to compare cohomology theories with different coefficients since the language of ${\mathbb T}_{\cal K}$ will contain one sort for the field of coefficients of the theory, which therefore varies with it.

Notice that if both $\ell$-adic \'etale cohomology and $p$-adic \'etale cohomology (for schemes in characteristic $p$) are $\cal K$ the theory ${\mathbb T}_{\cal K}$ will not be (geometrically) complete (if we axiomatize cohomology groups separately for each degree). Indeed, the bounds on the dimensions of cohomology groups are clearly expressible within geometric logic over the language of ${\mathbb T}_{\cal K}$ (as the requirement that any set of vectors with a size greater than the given bound be linearly dependent). On the other hand, it seems likely that there exists a complete theory $\mathbb M$ whose models are the $\ell$-adic cohomology theories (for all $\ell \neq p$) and the crystalline or rigid cohomology theory. This would imply the independence from $\ell$ not just of the dimensions of the cohomology groups, but of the dimensions of kernels and images of definable morphisms. For instance, for any $\lambda \in {\mathbb Q}$, the multiplicity $\textrm{dim}(T(f)-\lambda.Id)$ of $\lambda$ as an eigenvalue of the image $T(f)$ of a definable morphism $f$ by a (co)homology theory $T$ would be the same for all $T$ which are models of $\mathbb M$. More generally, for any rational polynomial $P\in {\mathbb Q}[X]$, the dimension $\textrm{dim}(\textrm{Ker}(P(T(f))))$ is independent from $T$.

Proving that a certain theory is complete is in general a hard matter; on the other hand, completeness of a geometric theory $\mathbb T$ can be reformulated as two-valuedness of its classifying topos and, if the topos is atomic, i.e. of the form $\Sh({\cal C}^{\textrm{op}}, J_{at})$ for some category $\cal C$ satisfying the amalgamation property (where $J_{at}$ is the atomic topology on it), this condition implies that all the models of $\mathbb T$ are conservative and is equivalent to the joint embedding property on the category $\cal C$, a condition that is much more amenable to a direct verification. 

We propose below a general framework implementing this approach for building \ac motivic toposes' based on the concept of atomic two-valued topos, and in particular on the topos-theoretic interpretation of Fra\"iss\'e's construction established in \cite{OC2}.

\section{The setting of atomic two-valued toposes}

In this section we argue that the setting of atomic two-valued toposes is a particularly suitable one for building \ac motivic toposes', also in relation to the syntactic interpretation of Nori's construction carried out in \cite{NoriSyntactic}. 

Recall that a Grothendieck topos is said to be \emph{atomic} if every object is a coproduct of atoms (i.e., of objects which do not have any proper subobjects).

Atomic two-valued toposes satisfy the following important property: any point of them is an open surjection (i.e. its inverse image is faithful and preserves exponentials and the subobject classifier). This is particularly relevant for our purposes since a property that is naturally expected of the realisations functors of a category of motives is that they preserve the rigid tensor structure present on it. Now, since internal tensor products of internal vector spaces over an internal field in a Grothendieck topos can be built by using finite limits and arbitrary colimits, and internal duals of internal vector spaces over an internal field can be built by only using finite limits and exponentials, any realisation functor coming from an open geometric morphism will preserve them.  

Moreover, pointed atomic two-valued toposes define a non-abelian lifting of the context of abelian categories with a faithful exact functor to the category of finite-dimensional vector spaces over a field. More specifically, as every $k$-linear abelian category $\cal A$ with a faithful functor $U:{\cal A}\to k\textrm{-vect}_{\textrm{fin}}$ is equivalent to the category $\textrm{Comod}_\textrm{fin}(\textrm{End}^{\vee}(U))$ of finite-dimensional comodules over the coalgebra of endomorphisms of $U$, so every atomic topos with a point $p$ is equivalent to the category $\Cont(\textrm{Aut}_{l}(p))$ of continuous actions of the localic automorphism group $\textrm{Aut}_{l}(p)$ of $p$ (cf.\cite{DubucG}). It is actually possible to associate to an atomic two-valued topos $\cal E$ with a point $p$ a $k$-linear abelian category ${\cal A}$ with a faithful exact functor $p_{\cal A}:{\cal A}\to k\textrm{-vect}_{\textrm{fin}}$ by taking $\cal A$ equal to the full subcategory $k\textrm{-vect}_{fin}({\cal E})$ of internal $k$-vector spaces $V$ in $\cal E$ such that $p^{\ast}(V)$ is a finite-dimensional vector space and $p_{{\cal A}}$ equal to the restriction of $p^{\ast}$ to $\cal A$. It would be interesting to characterize those pointed atomic two-valued toposes $({\cal E}, p)$ which can be recovered from the category $k\textrm{-vect}_{fin}({\cal E})$ or an appropriate subcategory of it, for instance as a completion of it or, by taking into account the rigid tensor structure present on $k\textrm{-vect}_{fin}({\cal E})$, as the classifying topos for (duals-preserving and tensor-preserving) representations of $k\textrm{-vect}_{fin}({\cal E})$ (into rigid $k$-linear tensor categories). 

For $k$-linear abelian categories with a faithful functor $U:{\cal A}\to k\textrm{-vect}$ to possibly infinite-dimensional vector spaces, this no longer works; the Tannaka-type construction of $\cal A$ as a category of comodules is replaced in \cite{NoriSyntactic} by a syntactic construction which generalizes it. The topos-theoretic analogue is provided by connected and locally connected toposes; indeed, any abelian category is equivalent to a full subcategory of the category of abelian objects in such a topos, namely the topos of regular sheaves on it (see also \cite{NoriSyntactic}).  
 
Lastly, atomic two-valued toposes constitute the natural setting for formulating the general Galois theory of topological or localic groups (cf. \cite{OCG}, \cite{DubucG} and section \ref{Fraisse} below). For any point $p$ of an atomic two-valued topos $\cal E$, $\cal E$ is equivalent to the topos of continuous actions of the \emph{localic automorphism group} $\textrm{Aut}_{l}(p)$ of $p$ (cf. \cite{DubucG}); if $p$ is universal and ultrahomogeneous (in the sense of \cite{OCG}) then $\cal E$ is equivalent to the topos of continuous actions of the \emph{topological automorphism group} $\textrm{Aut}_{t}(p)$ of $p$ (cf. \cite{OCG}). 

If $({\cal E}, p)$ is a pointed \ac motivic topos', we can define the \emph{motivic Galois group} of $p$ as the group of natural transformations of $p$ which respect the rigidity and the tensor structure on $k\textrm{-vect}_{fin}({\cal E})$. This is a smaller group with respect to $\textrm{Aut}_{l}(p)$ or $\textrm{Aut}_{t}(p)$, but it suffices, by Tannaka duality, to recover the Tannakian category $k\textrm{-vect}_{fin}({\cal E})$.

In order to investigate atomic toposes from the point of view of the theories that they classify, it is important to review the concept of theory of presheaf type. 

\subsection{Preliminaries on theories of presheaf type}

A geometric theory $\mathbb T$ is said to be \emph{of presheaf type} if it is classified by a presheaf topos. Recall that the classifying topos of a geometric theory $\mathbb T$ can always be canonically constructed as the category $\Sh({\cal C}_{\mathbb T}, J_{\mathbb T})$ of sheaves on the syntactic site $({\cal C}_{\mathbb T}, J_{\mathbb T})$ of $\mathbb T$ (cf. section \ref{sec:classifyingtoposes}). 

Every finitary algebraic theory (i.e., any theory whose axioms are of the form $(\top \vdash_{\vec{x}} t(\vec{x})=s(\vec{x}))$, where $t$ and $s$ are terms over the signature of $\mathbb T$ in the context $\vec{x}$) is of presheaf type, but the class of theories of presheaf type contains many other interesting mathematical theories (for instance, the coherent theory of linear orders or the infinitary theory of algebraic extensions of a given field). 

The following notions will be central for our purposes.

\begin{definition}
Let $\mathbb T$ be a geometric theory over a signature $\Sigma$ and $M$ a set-based $\mathbb T$-model. Then
\begin{enumerate}[(a)]
\item The model $M$ is said to be \emph{finitely presentable} if the representable functor $\textrm{Hom}_{{\mathbb T}\textrm{-mod}(\Set)}(M,-):{\mathbb T}\textrm{-mod}(\Set) \to \Set$ preserves filtered colimits;

\item The model $M$ is said to be \emph{finitely presented} if there is a geometric formula $\{\vec{x}. \phi\}$ over $\Sigma$ and a string of elements $(\xi_{1}, \ldots, \xi_{n})\in MA_{1}\times \cdots \times MA_{n}$ (where $A_{1}, \ldots, A_{n}$ are the sorts of the variables in $\vec{x}$), called the generators of $M$, such that for any $\mathbb T$-model $N$ in $\Set$ and string of elements $(b_{1}, \ldots, b_{n})\in MA_{1}\times \cdots \times MA_{n}$ such that $(b_{1}, \ldots, b_{n})\in [[\vec{x}.\phi]]_{N}$, there exists a unique arrow $f:M\to N$ in ${\mathbb T}\textrm{-mod}(\Set)$ such that $(fA_{1}\times \ldots \times fA_{n})((\xi_{1}, \ldots, \xi_{n}))=(b_{1}, \ldots, b_{n})$. 
\end{enumerate}

\end{definition}

The full subcategory of the category ${\mathbb T}\textrm{-mod}(\Set)$ of $\mathbb T$-models and model homomorphisms between them on the finitely presentable $\mathbb T$-models will be denoted by $\textrm{f.p.} {\mathbb T}\textrm{-mod}(\Set)$.

The classifying topos of a theory of presheaf type can always be canonically represented as the functor category $[\textrm{f.p.} {\mathbb T}\textrm{-mod}(\Set), \Set]$. 

For a theory of presheaf type $\mathbb T$, the two above-mentioned notions of finite presentability coincide. More specifically, we have the following result.

\begin{definition}
Let $\mathbb T$ be a geometric theory over a signature $\Sigma$. A geometric formula-in-context $\{\vec{x}. \phi\}$ is said to be \emph{$\mathbb T$-irreducible} if for any family $\{\theta_{i} \mid i\in I\}$ of $\mathbb T$-provably functional geometric formulae $\{\vec{x_{i}}, \vec{x}.\theta_{i}\}$ from $\{\vec{x_{i}}. \phi_{i}\}$ to $\{\vec{x}. \phi\}$ such that $(\phi \vdash_{\vec{x}} \mathbin{\mathop{\textrm{\huge $\vee$}}\limits_{i\in I}}(\exists \vec{x_{i}})\theta_{i})$ is provable in $\mathbb T$, there exist $i\in I$ and a $\mathbb T$-provably functional geometric formula $\{\vec{x}, \vec{x_{i}}. \theta'\}$ from $\{\vec{x}. \phi\}$ to $\{\vec{x_{i}}. \phi_{i}\}$ such that $(\phi \vdash_{\vec{x}} (\exists \vec{x_{i}})(\theta' \wedge \theta_{i}))$ is provable in $\mathbb T$. 
\end{definition}

\begin{theorem}\cite[Theorem 4.3]{OC7}\label{Onetoposmany_presheafcomplete}
Let $\mathbb T$ be a theory of presheaf type over a signature $\Sigma$. Then
\begin{enumerate}[(i)]
\item Any finitely presentable $\mathbb T$-model in $\Set$ is presented by a $\mathbb T$-irreducible geometric formula $\phi(\vec{x})$ over $\Sigma$;

\item Conversely, any $\mathbb T$-irreducible geometric formula $\phi(\vec{x})$ over $\Sigma$ presents a $\mathbb T$-model.
\end{enumerate}
In fact, the category $\textrm{f.p.} {\mathbb T}\textrm{-mod}(\Set)^{\textrm{op}}$ is equivalent to the full subcategory ${\cal C}_{\mathbb T}^{\textrm{irr}}$ of ${\cal C}_{\mathbb T}$ on the $\mathbb T$-irreducible formulae. 
\end{theorem}
  
The following theorem provides a syntactic criterion for a theory $\mathbb T$ to be of presheaf type, which amounts to the requirement that every geometric formula over the signature of $\mathbb T$ should be covered in the syntactic site of $\mathbb T$ by $\mathbb T$-irreducible formulas. 

\begin{theorem}\cite[Corollary 3.15]{OC7}\label{CriterionPresheafType}
Let $\mathbb T$ be a geometric theory over a signature $\Sigma$. Then $\mathbb T$ is of presheaf type if and only if there exists a collection $\cal F$ of $\mathbb T$-irreducible geometric formulae-in-context over $\Sigma$ such that for any geometric formula $\{\vec{y}. \psi\}$ over $\Sigma$ there exist objects $\{\vec{x_{i}}. \phi_{i}\}$ in $\cal F$ as $i$ varies in $I$ and $\mathbb T$-provably functional geometric formulae $\{\vec{x_{i}}, \vec{y}.\theta_{i}\}$ from $\{\vec{x_{i}}. \phi_{i}\}$ to $\{\vec{y}. \psi\}$ such that $(\psi \vdash_{\vec{y}} \mathbin{\mathop{\textrm{\huge $\vee$}}\limits_{i\in I}}(\exists \vec{x_{i}})\theta_{i})$ is provable in $\mathbb T$.
\end{theorem}

As an illustration of this theorem, consider the injectivization of the algebraic theory of Boolean algebras (see section \ref{sec_homoginj} below for the definition of injectivization of a geometric theory). In this theory, which is well-known to be of presheaf type, the formula $\{x. \top\}$ does not present a model (that is, there is not a free model on one generator), but it is covered by the formulae $\{x. x=0\}\mono \{x. \top\}$, $\{x. x=1\}\mono \{x. \top\}$ and $\{x. x\neq 0 \wedge x\neq 1\}\mono \{x. \top\}$, which are irreducible in the theory since they present respectively the models $\{0,1\}$, $\{0, 1\}$ and the four element Boolean algebra $\{0, 1, b, \neg b\}$ (notice that the formulae $\{x. x=0\}$ and $\{x. x=1\}$ are isomorphic in the syntactic category of the theory to the formula $\{[]. \top\}$, which presents the initial algebra $\{0,1\}$).

\begin{lemma}\label{lem_fp}
Let $\mathbb T$ be a theory of presheaf type and $\{\vec{x}.\phi\}$ a formula-in-context which presents a $\mathbb T$-model $M_{\{\vec{x}.\phi\}}$. Then for every geometric formula $\psi(\vec{x})$ over the signature of $\mathbb T$, the sequent $(\phi \vdash_{\vec{x}} \psi)$ is provable in $\mathbb T$ if and only if, denoting by $\vec{\xi}$ the set of generators of $M_{\{\vec{x}.\phi\}}$, $M_{\{\vec{x}.\phi\}} \vDash \psi(\vec{\xi})$.
\end{lemma}

\begin{proof}
Since $\mathbb T$ is of presheaf type, provability of geometric sequents in $\mathbb T$ is equivalent to validity in all $\Set$-based $\mathbb T$-models. For any set-based $\mathbb T$-model $M$, the sequent $(\phi \vdash_{\vec{x}} \psi)$ is valid in $M$ if and only if for any tuple $\vec{a}$ such that $M \vDash \phi(\vec{a})$, $M \vDash \psi(\vec{a})$. Now, by the universal property of the $\mathbb T$-model $M_{\{\vec{x}.\phi\}}$ presented by the formula $\{\vec{x}.\phi\}$, there exists a $\mathbb T$-model homomorphism $f:M_{\{\vec{x}.\phi\}}\to M$ such that $f(\vec{\xi})=\vec{a}$. Hence $M \vDash \psi(\vec{a})$ since $M_{\{\vec{x}.\phi\}} \vDash \psi(\vec{\xi})$ and $\psi$ is geometric. 
\end{proof}

Theorem \ref{CriterionPresheafType} and Lemma \ref{lem_fp} can be useful in connection with the problem of constructing a theory classified by a given topos $[{\cal C}, \Set]$. Indeed, according to the general method for constructing geometric theories classified by a given presheaf topos developed in 7.2 of \cite{OCPT}, the first step in constructing a theory classified by the topos $[{\cal C}, \Set]$ is to realise $\cal C$ as a full subcategory of the category of finitely presentable models of a theory $\mathbb S$ that one already knows to be of presheaf type; this will ensure that there is a quotient ${\mathbb S}_{\cal C}$ of $\mathbb S$ classified by the topos $[{\cal C}, \Set]$. The general theorems of section 6.4.2 of \cite{OCPT} take care of obtaining explicit axiomatizations for such a quotient. Theorem \ref{CriterionPresheafType} shows that all the sequents expressing the fact that every geometric formula over the signature of $\mathbb S$ is covered by the family of formulae presenting a model in $\cal C$ should be provable in ${\mathbb S}_{\cal C}$, while Lemma \ref{lem_fp} shows that all the sequents of the form $(\phi \vdash_{\vec{x}} \psi)$, where $\{\vec{x}.\phi\}$ is a formula presenting a model $c$ in $\cal C$ with generators $\vec{\xi}$ and $\psi(\vec{x})$ is a formula such that $c\vDash \psi(\vec{\xi})$, should be provable in ${\mathbb S}_{\cal C}$. These sequents should therefore be certainly added to $\mathbb S$ to form ${\mathbb S}_{\cal C}$. We shall apply these remarks in section \ref{sec:basictheorypresheaftype}.

\subsection{The topos-theoretic Fra\"iss\'e theorem}\label{Fraisse}

The following notions, which will play a central role in our analysis, are categorical generalisations of the concepts involved in the classical Fra\"iss\'e's construction (cf. \cite{OC2}).

\begin{definition}
A category $\mathcal C$ is said to satisfy the \emph{amalgamation property} (AP) if for every objects $a,b,c\in {\mathcal C}$ and morphisms $f:a\rightarrow b$, $g:a\rightarrow c$ in $\mathcal C$ there exists an object $d\in \mathcal C$ and morphisms $f':b\rightarrow d$, $g':c\rightarrow d$ in $\mathcal C$ such that $f'\circ f=g'\circ g$:
\[  
\xymatrix {
a \ar[d]_{g} \ar[r]^{f} & b  \ar@{-->}[d]^{f'} \\
c \ar@{-->}[r]_{g'} & d } 
\] 
\end{definition} 

Note that if $\mathcal C$ satisfies AP then we may equip ${\mathcal C}^{\textrm{op}}$ with the atomic topology, that is the Grothendieck topology whose covering sieves are exactly the non-empty ones. This point will be a fundamental ingredient of our topos-theoretic interpretation of Fra\"iss\'e's theorem. 

\begin{definition}
A category $\mathcal C$ is said to satisfy the \emph{joint embedding property} (JEP) if for every pair of objects $a,b\in {\mathcal C}$ there exists an object $c\in \mathcal C$ and morphisms $f:a\rightarrow c$, $g:b\rightarrow c$ in $\mathcal C$:
\[  
\xymatrix {
 & a  \ar@{-->}[d]^{f} \\
b \ar@{-->}[r]_{g} & c } 
\] 
\end{definition} 
Notice that if $\mathcal C$ has a weakly initial object (i.e., an object which admits an arrow to any other object of $\mathcal C$) then AP on $\mathcal C$ implies JEP on $\mathcal C$; however, in general the two notions are quite distinct from each other. 

\begin{definition}\label{defhomogeneity}
Let ${\mathcal C}\hookrightarrow {\mathcal D}$ be the embedding of a subcategory $\cal C$ into a category $\cal D$.
\begin{enumerate}[(a)]
\item An object $u\in {\mathcal D}$ is said to be \emph{$\mathcal C$-homogeneous} if for every objects $a,b \in {\mathcal C}$ and arrows $j:a\rightarrow b$ in ${\mathcal C}$ and $\chi:a\rightarrow u$ in $\mathcal D$ there exists an arrow $\tilde{\chi}:b\rightarrow u$ in $\mathcal D$ such that $\tilde{\chi}\circ j=\chi$:
\[  
\xymatrix {
a \ar[d]_{j} \ar[r]^{\chi} & u \\
b \ar@{-->}[ur]_{\tilde{\chi}} &  } 
\] 
\item An object $u\in {\mathcal D}$ is said to be \emph{$\mathcal C$-ultrahomogeneous} if for every objects $a,b \in {\mathcal C}$ and arrows $j:a\rightarrow b$ in ${\mathcal C}$ and $\chi_{1}:a\rightarrow u$, $\chi_{2}:b\rightarrow u$ in $\mathcal D$ there exists an isomorphism $\check{j}:u\rightarrow u$ such that $\check{j}\circ \chi_{1}=\chi_{2}\circ j$:
\[  
\xymatrix {
a \ar[d]_{j} \ar[r]^{\chi_{1}} & u \ar@{-->}[d]^{\check{j}}\\
b \ar[r]_{\chi_{2}} & u } 
\] 
\item An object $u\in {\mathcal D}$ is said to be \emph{$\mathcal C$-universal} if it is $\mathcal C$-cofinal, that is for every $a\in {\mathcal C}$ there exists an arrow $\chi:a\rightarrow u$ in $\mathcal D$:
\[  
\xymatrix {
a \ar@{-->}[r]^{\chi} & u  } 
\]    
\end{enumerate} 
\end{definition} 

\begin{remarks}
\begin{enumerate}[(a)]
\item It is easy to see that if $u$ is $\mathcal C$-ultrahomogeneous and $\mathcal C$-universal then $u$ is $\mathcal C$-homogeneous.

\item If $\cal C$ has an initial object in $\cal D$ then every $\cal C$-homogeneous object is $\cal C$-universal.

\item In verifying that an object $u$ in $\mathcal D$ is $\mathcal C$-ultrahomogeneous one can clearly suppose, without loss of generality, that the arrow $j$ in the definition is an identity.
\end{enumerate} 

\end{remarks}

Let $\mathbb T$ be a theory of presheaf type over a signature $\Sigma$ such that its category $\textrm{f.p.} {\mathbb T}\textrm{-mod}(\Set)$ of finitely presentable models satisfies the amalgamation property. Then we can put on the opposite category ${\textrm{f.p.} {\mathbb T}\textrm{-mod}(\Set)}^{\textrm{op}}$ the atomic topology $J_{at}$, obtaining a subtopos $\Sh({\textrm{f.p.} {\mathbb T}\textrm{-mod}(\Set)}^{\textrm{op}}, J_{at})$ of the classifying topos $[{\textrm{f.p.} {\mathbb T}\textrm{-mod}(\Set)}, \Set]$ of $\mathbb T$, which corresponds by the duality theorem of \cite{OCL} to a unique quotient ${\mathbb T}'$ of ${\mathbb T}$. 

This quotient can be characterized as the theory over $\Sigma$ obtained from $\mathbb T$ by adding all the sequents of the form $(\psi \vdash_{\vec{y}} (\exists \vec{x}) \theta(\vec{x}, \vec{y}))$, where $\phi(\vec{x})$ and $\psi(\vec{x})$ are formulae which present a $\mathbb T$-model and $\theta(\vec{x}, \vec{y})$ is a $\mathbb T$-provably functional formula from $\{\vec{x}. \phi\}$ to $\{\vec{y}. \psi\}$. Semantically, the theory ${\mathbb T}'$ axiomatizes the homogeneous models of $\mathbb T$, that is the models $M$ of $\mathbb T$ in $\Set$ such that for any arrow $y:c\to M$ in ${\mathbb T}\textrm{-mod}(\Set)$ and any arrow $f$ in $\textrm{f.p.} {\mathbb T}\textrm{-mod}(\Set)$ there exists an arrow $u$ in ${\mathbb T}\textrm{-mod}(\Set)$ such that $u \circ f=y$:
\[  
\xymatrix {
c \ar[d]_{f} \ar[r]^{y} & M \\
d \ar@{-->}[ur]_{u} &  } 
\] 
  
For this reason, we shall call ${\mathbb T}'$ the `theory of homogeneous $\mathbb T$-models'. Notice that a $\mathbb T$-model $M$ is homogeneous if and only if it is $\textrm{f.p.} {\mathbb T}\textrm{-mod}(\Set)$-homogeneous as an object of the category ${\mathbb T}\textrm{-mod}(\Set)$ in the sense of Definition \ref{defhomogeneity}. 

\begin{theorem}\cite{OC2}\label{mainFraisse}
Let $\mathbb T$ be a theory of presheaf type such that the category $\textrm{f.p.} {\mathbb T}\textrm{-mod}(\Set)$ is non-empty and satisfies AP and JEP. Then the topos 
\[
\Sh({\textrm{f.p.} {\mathbb T}\textrm{-mod}(\Set)}^{\textrm{op}}, J_{at})
\]
is atomic and two-valued; in other words, the theory ${\mathbb T}'$ of homogeneous $\mathbb T$-models is complete and atomic. In particular, any two models of ${\mathbb T}'$ satisfy exactly the same first-order formulae over the signature of $\mathbb T$.
\end{theorem}

\begin{remark}
The idea underlying this topos-theoretic generalization of Fra\"is-s\'e's theorem is that two homogeneous models are not necessarily comparable with each other directly, but since they can both be represented as filtered colimits of finitely presentable models, links between them can be established by working at the level of these `small' models, which can be amalgamated with each other thanks to the AP and JEP. 
\end{remark}

The following result shows that the toposes arising in Theorem \ref{mainFraisse} often admit Galois-type representations. 

\begin{theorem}\cite{OCG}
\label{galois}
Let $\mathbb T$ be a theory of presheaf type such that its category $\textrm{f.p.}{\mathbb T}\textrm{-mod}(\Set)$ of finitely presentable models satisfies AP and JEP, and let $M$ be a $\textrm{f.p.}{\mathbb T}\textrm{-mod}(\Set)$-universal and $\textrm{f.p.}{\mathbb T}\textrm{-mod}(\Set)$-ultrahomogeneous model of $\mathbb T$. Then we have an equivalence of toposes 
\[
\Sh({\textrm{f.p.}{\mathbb T}\textrm{-mod}(\Set)}^{\textrm{op}}, J_{at})\simeq \Cont(\textrm{Aut}(M)),
\]
where $\textrm{Aut}(M)$ is endowed with the topology of pointwise convergence (in which a basis of open neighbourhoods of the identity is given by the sets of the form $\{f:M\cong M \mid f(\vec{a})=\vec{a}\}$ for any $\vec{a}\in M$), which is induced by the functor
\[
F:{\textrm{f.p.}{\mathbb T}\textrm{-mod}(\Set)}^{\textrm{op}}\to \Cont(\textrm{Aut}(M)) 
\]
sending any model $c$ of $\textrm{f.p.}{\mathbb T}\textrm{-mod}(\Set)$ to the set $\textrm{Hom}_{{\mathbb T}\textrm{-mod}(\Set)}(c, M)$ (endowed with the obvious action by $\textrm{Aut}(M)$) and any arrow $f:c\to d$ in $\textrm{f.p.}{\mathbb T}\textrm{-mod}(\Set)$ to the $\textrm{Aut}(M)$-equivariant map
\[
-\circ f:\textrm{Hom}_{{\mathbb T}\textrm{-mod}(\Set)}(d, M)\to \textrm{Hom}_{{\mathbb T}\textrm{-mod}(\Set)}(c, M).
\]
\end{theorem}

\begin{remark}\label{rem_independence}
One can have in general many different $M$ with non-isomorphic automorphism groups $\textrm{Aut}(M)$ whose associated toposes $\Cont(\textrm{Aut}(M))$ are equivalent (this phenomenon of course does not only hold in the topological setting but also in the localic one). Take for instance the Schanuel topos, that is the topos of Theorem \ref{galois} where $\mathbb T$ is the injectivization of the empty theory over a one-sorted signature; by the theorem, it can be represented as $\Cont(\textrm{Aut}(M))$ for \emph{any} infinite set $M$. This is in fact a quite deep phenomenon, which shows that one cannot directly relate the different $M$'s which each other without using the associated topos as a `bridge'.

We would actually like to interpret the relationships between $\ell$-adic cohomologies for different $\ell$'s as arising from the fact that they yield different representations (as continuous actions of their respective automorphism groups) of the same atomic two-valued \ac motivic topos' (cf. section \ref{sec:motivictheories}).  
\end{remark}

The following proposition illuminates the relationship between homogeneity of a model $M$ and surjectivity of the maps $Mf$ (for $f$ a function symbol over the signature of the theory).  

\begin{proposition}
Let $\mathbb T$ be a theory of presheaf type over a signature $\Sigma$, $f:A_{1}, \ldots, A_{n}\to B$ a function symbol over $\Sigma$ and $\psi(y^{B})$ a formula presenting a $\mathbb T$-model such that the sequent $(\psi(f(\vec{x})) \vdash_{\vec{x}} \bot)$ is not provable in $\mathbb T$. Then the sequent $(\psi \vdash_{y} (\exists \vec{x})f(\vec{x})=y)$ is provable in the theory ${\mathbb T}'$ of homogeneous $\mathbb T$-models; in particular, for any homogeneous $\mathbb T$-model $M$, the map $Mf:MA_{1}\times \cdots \times MA_{n}\to MB$ is surjective onto $[[y. \psi]]_{M}$.   
\end{proposition}

\begin{proof}
Since the formula $\{\vec{x}.\psi(f(\vec{x})\}$ is not $\mathbb T$-provably equivalent to $\{\vec{x}. \bot\}$, there exists a non-empty covering of it in the syntactic category by $\mathbb T$-irreducible formulae (cf. Theorem \ref{CriterionPresheafType}). In particular, there exists an arrow $[\theta]:\{\vec{z}. \chi\}\to \{\vec{x}.\psi(f(\vec{x})\}$ in ${\cal C}_{\mathbb T}$ where the formula-in-context $\{\vec{z}. \chi\}$ presents a $\mathbb T$-model. By composing $[\theta]$ with the arrow $[y=f(\vec{x}) \wedge \psi(y))]:\{\vec{x}.\psi(f(\vec{x})\} \to \{y. \psi\}$ we obtain the arrow
\[
[\tau(\vec{z}, y)]:=[(\exists \vec{x})(\theta(\vec{z}, \vec{x}) \wedge y=f(\vec{x}) \wedge \psi(y))]:\{\vec{z}. \chi\}\to \{y. \psi\}.   
\]
By the syntactic description of the theory of homogeneous $\mathbb T$-models given above, it follows that the sequent $(\psi \vdash_{y} (\exists \vec{z})\tau(\vec{z}, y))$ is provable in ${\mathbb T}'$; but this sequent clearly entails the sequent $(\psi \vdash_{y} (\exists \vec{x})f(\vec{x})=y)$, whence our thesis follows.
\end{proof}

\subsubsection{Homogeneity and injectivizations}\label{sec_homoginj}

As shown by the following theorem, the notion of homogeneous $\mathbb T$-model essentially trivializes in the case of a theory $\mathbb T$ in which the formulae $\{\vec{x}. \top\}$ present a model (which is always the case if $\mathbb T$ is algebraic) or, more generally, if the formulae $\psi(\vec{x})$ which present a $\mathbb T$-model are satisfied by `too many' elements. A simple way for preventing such a situation is to make all the $\mathbb T$-model homomorphisms sortwise injective and then consider a presheaf completion of the resulting theory. Recall from \cite{OCPT} that a \emph{presheaf completion} of a geometric theory $\mathbb T$ is a theory of presheaf type given by an expansion $\mathbb S$ of $\mathbb T$ which is classified by the topos $[\textrm{f.p.} {\mathbb T}\textrm{-mod}(\Set), \Set]$; in particular, the finitely presentable $\mathbb S$-models can be identified with the finitely presentable $\mathbb T$-models if the category of set-based $\mathbb T$-models is finitely accessible (i.e., equivalent to the ind-completion of some small category). A given geometric theory has many different presheaf completions in general (and it always has one), but all of them are Morita-equivalent, that is they have the same classifying topos.

In fact, the notion of homogeneous model is mostly relevant when all the $\mathbb T$-model homomorphisms are sortwise injective. The following construction turns a geometric theory $\mathbb T$ into a geometric theory ${\mathbb T}_{i}$ whose set-based models can be identified with those of $\mathbb T$ and whose model homomorphisms are precisely the sortwise injective $\mathbb T$-model homomorphisms.

\begin{definition}\label{injectivization}
Let $\mathbb T$ be a geometric theory over a signature $\Sigma$. The \emph{injectivization} ${\mathbb T}_{i}$ of $\mathbb T$ is the geometric theory obtained from $\mathbb T$ by adding a binary predicate $D_{A}\mono A, A$ for each sort $A$ over $\Sigma$ and the coherent sequents
\[
(D_{A}(x^{A}, y^{A}) \wedge x^{A}=y^{A}) \vdash_{x^{A}, y^{A}} \bot)
\]
and 
\[
(\top \vdash_{x^{A}, y^{A}} D_{A}(x^{A}, y^{A}) \vee x^{A}=y^{A}).
\]
\end{definition}

\begin{theorem}\label{thm_triviality}
Let $\mathbb T$ be a theory of presheaf type with the property that its category $\textrm{f.p.}{\mathbb T}\textrm{-mod}(\Set)$ of finitely presentable models satisfies AP, and ${\mathbb T}'$ be the theory of homogeneous $\mathbb T$-models. Then
\begin{enumerate}[(i)]
\item If the category $\textrm{f.p.}{\mathbb T}\textrm{-mod}(\Set)$ has an initial object $I$ then for any geometric formula $\psi(\vec{x})$ presenting a $\mathbb T$-model and any tuple $\vec{c}$ of constants of the same kind as $\vec{x}$ such that $I \vDash \psi(\vec{c})$, the sequent $(\psi \vdash_{\vec{x}} \vec{x}=\vec{c})$ is provable in ${\mathbb T}'$;

\item If there exist both the free $\mathbb T$-model on one generator and the free $\mathbb T$-model on two generators for a given sort $A$ over the signature of $\mathbb T$ then the sequent $(\top \vdash_{x^{A}, y^{A}} x^{A}= y^{A})$ is provable in the theory ${\mathbb T}'$. 
\end{enumerate}
\end{theorem}

\begin{proof}
(i) If the formula $\{[]. \top\}$ presents a model and $I \vDash \psi(\vec{c})$ then by Lemma \ref{lem_fp} the sequent $(\top \vdash_{[]} \psi(\vec{c}))$ is provable in $\mathbb T$. We thus have an arrow $[\vec{x}=\vec{c}]: \{[]. \top\}\to \{\vec{x}. \psi\}$ in the syntactic category of $\mathbb T$, whence by the syntactic description of ${\mathbb T}'$ given above the sequent $(\psi \vdash_{\vec{x}} \vec{x}=\vec{c})$ is provable in ${\mathbb T}'$.

(ii) If both the formulae $\{x^{A}. \top\}$ and $\{y^{A}, z^{A}. \top\}$ present a $\mathbb T$-model then the sequent corresponding to the arrow $[z=x \wedge y=x]:\{x^{A}. \top\} \to \{y^{A}, z^{A}. \top\}$, namely $(\top \vdash_{y, z} (\exists x)(y=x \wedge z=x)$, which is provably equivalent to the sequent $(\top \vdash_{x^{A}, y^{A}} x^{A}= y^{A})$, is provable in ${\mathbb T}'$, as required. 
\end{proof}

As an illustration of these results, consider the algebraic theory $\mathbb T$ of Boolean algebras. The theory of homogeneous $\mathbb T$-models is trivial, that is its unique model is the zero Boolean algebra in which $0=1$ (take $c=0$ and $\psi=\top$ in the theorem). On the other hand, the theory of homogeneous ${\mathbb T}_{i}$-models is a very interesting one; indeed, its models are precisely the atomless Boolean algebras. The algebra $\{0,1\}$ constitutes the initial object $I$ of the category of ${\mathbb T}_{i}$-models, but the free models on one or two generators no longer exist in this theory. For instance, the 4-element Boolean algebra is no longer presented by the formula $\{x. \top\}$, but it is presented by the formula $\{x. x\neq 0 \wedge x\neq 1\}$, which in fact no longer holds in $I$ when evaluated at some constant (either $0$ or $1$) over the signature of $\mathbb T$.

\section{Motivic theories}\label{sec:motivictheories}

\subsection{The general strategy}\label{sec:generalstrategy}

Our strategy for building an atomic two-valued topos classifying (co)homology theories belonging to a certain class $\cal K$ is as follows.

Once chosen a signature $\Sigma_{\cal K}$ as in section \ref{sec:topdown} so that all the (co)homology theories in $\cal K$ can be regarded as $\Sigma_{\cal K}$-structures, one should define a theory ${\mathbb S}_{\cal K}$ which axiomatizes precisely the $\Sigma_{\cal K}$-substructures of the (co)homology theories in $\cal K$. Indeed, if we want the (co)homology theories in $\cal K$ to be ${\cal C}$-universal (and note that universality follows from homogeneity in presence of an initial object), where $\cal C$ is the category of finitely presentable models of the injectivization ${\mathbb S}^{i}_{\cal K}$ of ${\mathbb S}_{\cal K}$, the theory ${\mathbb S}^{i}_{\cal K}$ is \emph{forced} (up to presheaf completion) to be the theory which axiomatizes the ${\mathbb S}_{\cal K}$-substructures of a (co)homology theory in $\cal K$.

It is important to observe that the exactness conditions
\[
(g(x)=0 \vdash_{x} (\exists y)f(y)=x),
\]
for all distinguished pairs (i.e. pairs fitting in a long exact sequence) $(f, g)$ are satisfied by all the homology theories in $\cal K$ but are \emph{not} inherited by their $\Sigma_{\cal K}$-substructures. On the other hand, any algebraic sequent (i.e., any sequent whose premise is a finite conjunction of formulae of the form $w(\vec{x})=0$ and whose conclusion is a formula of the same form) which is satisfied by the (co)homology theories in $\cal K$ is valid in every $\Sigma_{\cal K}$-substructure of a (co)homology theory in $\cal K$.

It is thus natural to define ${\mathbb S}_{\cal K}$ as the set of algebraic sequents over $\Sigma_{\cal K}$ which are satisfied by all the (co)homology theories in $\cal K$ (or which follow from the usual axioms for (co)homology theories, as suitably formalized within geometric logic over $\Sigma_{\cal K}$). One should actually define ${\mathbb S}_{\cal K}$ to consist of \emph{all} the sequents over $\Sigma_{\cal K}$ which are satisfied by $\Sigma_{\cal K}$-substructures of theories in $\cal K$ (including in particular all the geometric sequents over $\Sigma_{\cal K}$ which are satisfied by the theories in $\cal K$ and whose conclusion does not contain any existential quantification), but there are technical reasons to believe that the algebraic sequents suffice (in the sense that they entail all the others).

It is natural to wonder whether it makes sense to look for a simple, explicit axiomatization of the theory ${\mathbb S}_{\cal K}$. To this end, we observe that the following axiom schemes exhibit sets of algebraic sequents which are derivable from the above-mentioned exactness conditions.

Below we say that an arrow $s$ in the syntactic category of ${\mathbb S}_{\cal K}$ factors through a distinguished pair if there exists a distinguished pair $(f, g)$ and an arrow $t$ such that $s = g\circ f \circ t$ in ${\mathbb S}_{\cal K}$.

\begin{enumerate}[(i)]
\item \emph{Axiom scheme E1:}
\[
(g(x)=0 \vdash_{x} s_{1}(x)=s_{2}(x)),
\]
for any distinguished pair $(f,g)$, where $f:c\to d$ and $g:d\to e$, and any $s_{1}, s_{2}:d \to e'$ such that $s_{1}\circ f=s_{2}\circ f$.

\item \emph{Axiom scheme E2:}
\[
(g(x)=0 \vdash_{x} s(x)=0),
\]
for any distinguished pair $(f,g)$, where $f:c\to d$ and $g:d\to e$, and any $s:d \to e'$ such that $s\circ f$ factors through a distinguished pair.

\item \emph{Axiom scheme E3:}
\[
(g_{0}(x)=0 \wedge g(\chi(x))=t(x) \vdash_{x} h(\chi(x))=t'(x))
\]
for any distinguished pairs $(f_{0}, g_{0})$ and $(f, g)$ and arrows $p, \chi, h, t, t'$ such that $h\circ f$ factors through a distinguished pair and $g\circ p=t\circ f_{0}$, $h\circ p=t'\circ f_{0}$.
\end{enumerate}

The fact that axiom schemes E1 and E2 are derivable from the exactness conditions is clear. It is instructive to verify that also the more complicated axiom scheme E3 follows from them. If $x=f_{0}(y)$ then $g(\chi(x))=t(x)$ implies that $g(\chi(f_{0}(y)))=t(f_{0}(y))=g(p(y))$. Therefore the element $\chi(f_{0}(y)-p(y))$ is in the kernel of $g$, whence by E2 $h(\chi(f_{0}(y)-p(y)))=0$, that is $h(\chi(f_{0}(y))=h(p(y)))$. But $h(p(y))=t'(f_{0}(y))$, whence $h(\chi(f_{0}(y)))=t'(f_{0}(y))$, that is $h(\chi(x))=t'(x)$, as required.

The existence of such complicated axioms which are derivable from the exactness conditions shows that there is not much hope of finding an explicit (finite) set of sequents axiomatizing the theory ${\mathbb S}_{\cal K}$.  

Notice that the theory ${\mathbb S}_{\cal K}$ should prove any sequent of the form $(\top \vdash_{x^{\name{X}}} x=0)$ for each $X$ such that $H(X)=0$ for every $H$ in $\cal K$. 

To get a non-trivial theory of homogeneous ${\mathbb S}_{\cal K}$-models (cf. Theorem \ref{thm_triviality}), we consider the injectivization ${\mathbb S}_{\cal K}^{i}$ of the theory ${\mathbb S}_{\cal K}$, whose category of set-based models has as objects the ${\mathbb S}_{\cal K}$-models and as arrows the injective homomorphisms between them.

The category $\cal C$ of finitely presentable models of ${{\mathbb S}_{\cal K}}^{i}$ of the theory ${\mathbb S}_{\cal K}$ should satisfy the amalgamation property and possibly have an initial object (so that the joint embedding property is automatically satisfied and universality follows from homogeneity).

Given the form of the axioms of ${\mathbb S}_{\cal K}$, it follows from Propositions 6.11 and 6.13 \cite{OCPT} that the finitely presentable ${\mathbb S}_{\cal K}^{i}$-models are precisely the \emph{finitely generated} ones. Recall that a model $M$ of a many-sorted theory is said to be finitely generated if there exists a finite set of elements of the sets $MA$ (where $A$ is a sort over the signature of the theory) such that any other elements can be obtained by applying terms written in the signature of the theory to this set of generators (notice that this is a global condition which does not imply that the model be finitely generated \emph{sortwise}). 

In order to apply Theorem \ref{mainFraisse}, one needs to turn the theory ${\mathbb S}_{\cal K}^{i}$ into a theory $\tilde{{\mathbb S}_{\cal K}^{i}}$ of presheaf type classified by the topos $[{\cal C}, \Set]$. By the results of section 7.2 of \cite{OCPT}, this can be achieved, by adding predicates $R$ for presenting all the models in $\cal C$, and disjunctive axioms involving these $R$ which fix their interpretations in any set-based model of $\tilde{{\mathbb S}^{i}_{\cal K}}$ as the set-theoretic complement of the interpretation of some geometric formula over the signature of ${\mathbb S}^{i}_{\cal K}$. The \emph{set-based models} of ${\mathbb S}_{\cal K}^{i}$ can thus be canonically and uniquely endowed with the structure of $\tilde{{\mathbb S}^{i}_{\cal K}}$-models, so to have an equivalence between the category of set-based ${\mathbb S}_{\cal K}^{i}$-models and that of set-based $\tilde{{\mathbb S}^{i}_{\cal K}}$-models (see section \ref{sec:basictheorypresheaftype} for an exemplification of these remarks). For more on presheaf completions of geometric theories, see section 7.2 of \cite{OCPT}. 

The \emph{motivic theory of $\cal K$} will be the theory of homogeneous $\tilde{{{\mathbb S}^{i}_{\cal K}}}$-models. By Theorem \ref{mainFraisse}, this is an atomic and complete theory, and any two of its models satisfy exactly the same first-order properties. The fact that $\ell$-adic and $\ell'$-adic cohomology groups have the same dimension over the respective coefficients fields will follow as a consequence of this, if we succeed in proving that both are homogeneous $\tilde{{{\mathbb S}^{i}_{\cal K}}}$-models. 

Notice that for all the theories in $\cal K$ to be  homogeneous as $\tilde{{{\mathbb S}^{i}_{\cal K}}}$-models, it is necessary that for any sort $\name{X}$ over the signature $\Sigma_{{\cal K}}$ and any $H$ and $H'$ in $\cal K$, $H(X)=0$ if and only if $H'(X)=0$. Indeed, if $\cal C$ has an initial object then every homogeneous $\tilde{{{\mathbb S}^{i}_{\cal K}}}$-model is universal, and injective homomorphisms preserve and reflect equalities. In other words, the possible reasons why a homogeneous (co)homology theory might be zero at a certain sort should be encoded in our basic theory of presheaf type $\tilde{{{\mathbb S}^{i}_{\cal K}}}$. In fact, this is perfectly in line with the general philosophy of motives which prescribes that these reasons should be \ac geometric', whence formalizable over the signature of $\tilde{{{\mathbb S}^{i}_{\cal K}}}$.
 
It is worth to note, for the sake of comparison between this \ac top-down' approach based on atomic two-valued toposes and the \ac bottom-up' one implemented in \cite{NoriSyntactic} that, given a $\cal C$-universal and $\cal C$-ultrahomogeneous homology theory $T:N \to {\mathbb Q}\textrm{-vect}$, the $\mathbb Q$-linear category of motives ${\cal C}_{T}$ associated to it as in \cite{NoriSyntactic} is a full subcategory of the topos $\Cont(\textrm{Aut}_{t}(T))$.

In the following section we identify a few axioms, expressible in geometric logic, which are satisfied by all the usual homology theories and their substructures. We should mention that many different approaches have been proposed to the problem of axiomatizing (co)homology theories (for instance, the Weil or Bloch-Ogus axioms). As remarked above, the theory ${{\mathbb S}_{\cal K}}$ should be obtained by retaining only the axioms which are inherited by $\Sigma_{\cal K}$-substructures.

\subsection{Some axioms for homology theories}\label{sec:someaxioms}

To fix ideas, to define our signature $\Sigma_{{\cal K}}$ let us start with the signature $\Sigma_{N}$ of \emph{Nori's quiver} $N_{k}$ (where $k$ is the base field), having one sort $\name{c}$ for each vertex $c$ of $N$, one function symbol $\name{f}:\name{c}\to \name{d}$ for each edge $f:c\to d$ of $N$ and the abelian group structure on each sort $\name{c}$.
 
Recall that $N_{k}$ has as objects the triplets $(X, Y, i)$ where $X$ is a variety over $k$, $Y$ is a closed subvariety of $X$ and $i$ is an integer, and has two different kinds of edges:
\begin{enumerate}[(1)]
\item (functoriality) For every morphism (of schemes) $f:X\to X'$ such that $f(Y)\subseteq Y'$ and every integer $i$, an arrow 
\[
p^{i}_{f:(X, Y)\to (X', Y')}:(X, Y, i)\to (X', Y', i);
\]

\item (boundary) For every integer $i$ and every chain $Z\subseteq Y \subseteq X$ of closed subschemes of $X$, an arrow
\[
\delta_{(X, Y, Z, i)}:(X, Y, i)\to (Y, Z, i-1).
\]
\end{enumerate}
   
Starting from the empty theory over the signature $\Sigma_{N}$, let us expand it by imposing the (coherent) structure of a field of characteristic $0$ on the sort $\name{0}=\name{(\textrm{Spec}(k), \emptyset, 0)}$ (notice that the property of having characteristic $0$ can be formalized by the set of sequents $(p.1=0 \vdash \bot)$ for all primes $p$) and the structure of a vector space over this field on all the other sorts $\name{c}$ (this is formalized by taking one binary function symbol $\name{0}\times \name{c}\to \name{c}$ for each $c$ in $N_{k}$). The intended interpretation of $\name{0}$ is the field of coefficients of the homology theory. It is natural to axiomatize the notion of a field by using the signature for von Neumann regular rings, so that the notion of finitely generated field coincides with the model-theoretic notion of finitely generated structure for models of this theory.

An important property of homology theories is that for any disjoint sum $X \coprod Y$ of two varieties $X$ and $X'$ and subvarieties $Y$ and $Y'$, $H(X\coprod X', Y\coprod Y', i)$ is isomorphic to the direct sum $H(X, Y, i)\oplus H(X', Y', i)$. This property is an easy consequence of exactness if one assumes the following property: for any pair $(X, Y)$ and any open $Z$ of $X$ contained in $Y$, the canonical homomorphisms $H(X\setminus Z, Y\setminus Z, i)\to H(X, Y, i)$ are isomorphisms. It is natural to enrich the signature of our theory so to name the inverses to this canonical homomorphisms and add the equational axioms specifying that they are inverses to them. Thanks to these axioms, when treating $0$-motives, we will be able to restrict our attention to varieties which consist of just one point defined over $k$ (that is, to the varieties of the form $\textrm{Spec}(k')\to \textrm{Spec}(k)$ for a finite field extension $k'$ of $k$), since every variety of dimension $0$ is a finite disjoint union of varieties of this form. 

The following axioms are satisfied by all homology theories and inherited by substructures of them, so they are provable in ${\mathbb S}_{\cal K}$.

\emph{Functoriality}:
\[
(\top \vdash_{x} \name{p_{g}}(\name{p_{f}}(x))=\name{p_{h}}(x)) 
\]
for any morphisms of schemes $f$ and $g$ with composite $h$; and
\[
(\top \vdash_{x} \name{p_{f}}(x)=x)
\]
for $f$ equal to an identity morphism on a scheme. 

\emph{Border naturality}:
\[
(\top \vdash_{x} \name{p^{n-1}_{f:(Y, Z)\to (Y', Z')}}(\name{\delta_{(X, Y, Z, n)}}(x))=\name{\delta_{(X', Y', Z',n)}}(\name{p^{n}_{f:(X, Y)\to (X', Y')}}(x)))
\]
for any closed subschemes $Z \subseteq Y \subseteq X$ and any morphism of schemes $f:(X, Y, Z, n)\to (X', Y', Z', n)$;

\emph{Complex condition}:
\[
(\top \vdash_{x} \name{p_{g}}(\name{p_{f}}(x)))=0
\]
for any distinguished pair $(f,g)$ of arrows in Nori's quiver $N_{k}$.

Notice that, whilst the complex condition is algebraic, it is not natural, both from a logical and an algebraic point of view, to put the exactness conditions among the axioms of our basic theory of presheaf type. Indeed, there are too few exact subsequences of given homology exact sequence, and the theory of exact sequences does not look at all like a theory of presheaf type. This means in particular that we should recover the exactness requirement as a particular case of homogeneity. This is indeed the case, as shown by Theorem \ref{thm_exacthom} (for the theory $\mathbb I$ defined below).

Finally, we add some axioms concerning the connected components of varieties containing a rational point, which ensure the existence of an initial object in the category of models of our theory (cf. Lemma \ref{initialobject} below). 
 
To this end, we observe that there are exactly three different kinds of edges in Nori's quiver $N_{k}$ whose source or target are the triplet $(\textrm{Spec}(k), \emptyset, 0)$:
\begin{enumerate}[(1)]
\item $p_{f_{x}}:(\textrm{Spec}(k), \emptyset, 0)\to (X, Y, 0)$ for any point $x$ of $X$ over $\mathbb Q$ (where $f_{x}$ is the morphism of schemes $\textrm{Spec}(k) \to X$ corresponding to the point $x$);

\item $p_{!_{X}}:(X, \emptyset, 0)\to (\textrm{Spec}(k), \emptyset, 0)$, where $!_{X}$ is the unique morphism $X\to \textrm{Spec}(k)$;

\item $\delta_{(X, \textrm{Spec}(k), \emptyset, 1)}:(X, \textrm{Spec}(k), 1) \to (\textrm{Spec}(k), \emptyset, 0)$.
\end{enumerate}

Notice the following relations: 
\[
p_{!_{X}}\circ p_{f_{x}}=id_{(\textrm{Spec}(k), \emptyset, 0)},
\]
for any $k$-point $x$ of $X$.

\begin{lemma}\label{connectedcomponents}
Let $X$ be a variety over $k$. Then $X$ has a finite number of connected components containing a rational point and for every homology theory $H$, denoting by $H_{f_{x}}:H(\textrm{Spec}(k), \emptyset, 0)\to H(X, \emptyset, 0)$ the $\mathbb Q$-linear map in homology induced by a rational point $x$, the value $H_{f_{x}}(1)$ (where $1$ is the unit of the ring $H(\textrm{Spec}(k), \emptyset, 0)$) only depends on the connected component of $X$ containing $x$ and the values $H_{f_{x_{1}}}(1), \ldots, H_{f_{x_{n}}}(1)$, where $x_{1}, \ldots, x_{n}$ are representatives for the connected components of $X$ containing a rational point, are linearly independent over $\mathbb Q$.
\end{lemma}\qed

The following axioms, which we call the \emph{axioms for connected components}, are satisfied by every homology theory (cf. Lemma \ref{connectedcomponents}):
\[
(\top \vdash_{[]} \name{p_{f_{x}}}(1)=\name{p_{f_{x'}}}(1))
\]
for any rational points $x, x'$ of $X$ lying in the same connected component of $X$, plus the infinite sequence of geometric axioms asserting the linear independence over $\mathbb Q$ of the values $\name{p_{f_{x_{1}}}}(1), \ldots, \name{p_{f_{x_{d}}}}(1)$ for any set $\{x_{1}, \ldots, x_{d}\}$ of representatives of connected components of $X$ containing a rational point.

\section{The basic theory of presheaf type}\label{sec:basictheorypresheaftype}

\subsection{Definition of the theory $\mathbb I$}\label{sec:theoryI}

In order to implement the strategy described in section \ref{sec:generalstrategy}, we start by investigating the theory $\mathbb I$ obtained by injectifying the theory over an expansion of the signature $\Sigma_{N}$ obtained by adding to all the axioms considered in section \ref{sec:someaxioms} and all the algebraic sequents 
\[
(w_{1}(\vec{x})=0 \wedge \cdots \wedge w_{n}(\vec{x})=0 \vdash_{\vec{x}} w(\vec{x})=0).
\]
which are provable in the theory obtained by adding to the axioms of section \ref{sec:someaxioms} the exactness conditions for all the distinguished pairs in Nori's quiver. 

The theory of homogeneous $\tilde{\mathbb I}$-models, where $\tilde{\mathbb I}$ is a presheaf completion of the theory $\mathbb I$, is not intended to be the final one whose classifying topos is a \ac motivic topos' with the desired properties, but should be regarded as a good \ac approximation' to it.

\begin{lemma}\label{initialobject}
The category $\textrm{f.p.} {\mathbb I}\textrm{-mod}(\Set)$ has an initial object.  
\end{lemma}

\begin{proof}
Let $I$ be the ${\mathbb I}$-model obtained by interpreting $(\textrm{Spec}(k), \emptyset, 0)$ as $\mathbb Q$, every triplet $(X, \emptyset, 0)$ such that $X$ has a point over $\mathbb Q$ as ${\mathbb Q}^{d}$, where $d$ is the number of connected components of $X$ containing a rational point, $p_{f_{x}}$ as the $\mathbb Q$-linear map ${\mathbb Q}\to {\mathbb Q}^{d}$ sending $1$ to the $d$-uple $(0, \ldots, 0 , 1, 0, \ldots, 0)$, where the $1$ is in the $i$-th position, $i$ being the cardinal of the component of $X$ containing $x$, $p_{!_{X}}$ as the $\mathbb Q$-linear map ${\mathbb Q}^{d}\to {\mathbb Q}$ sending any basis vector $(0, \ldots, 0 , 1, 0, \ldots, 0)$ to $1\in {\mathbb Q}$, and every other vertex and edge of $N_{k}$ as $0$.

It is an immediate consequence of Lemma \ref{connectedcomponents} that $I$ is initial in ${\mathbb I}\textrm{-mod}(\Set)$. The model $I$ actually lies in $\textrm{f.p.} {\mathbb I}\textrm{-mod}(\Set)$ since it is pointwise finitely generated. 
\end{proof}

\begin{remark}\label{rem:sortwisemodels}
For any field $K$ of characteristic zero, we have a model ${\cal I}_{K}$ in ${\mathbb I}\textrm{-mod}(\Set)$ defined by replacing $\mathbb Q$ with $K$ in the definition of the initial object of $\textrm{f.p.} {\mathbb I}\textrm{-mod}(\Set)$ given in the proof of Lemma \ref{initialobject}. Notice that ${\cal I}_{K}$ is finitely generated as a $\mathbb I$-model, that is it lies in $\textrm{f.p.} {\mathbb I}\textrm{-mod}(\Set)$, if and only if $K$ is a finitely generated field.
\end{remark}

We have to prove that the category $\textrm{f.p.} {\mathbb I}\textrm{-mod}(\Set)$ satisfies the AP. This will imply that it also satisfies the JEP (since it has an initial object by Lemma \ref{initialobject}). This will be done by using logical techniques in section \ref{sec_amalgamation}.

\subsection{A presheaf completion for $\mathbb I$}\label{sec:presheafcompletion}

In this section we shall explicitly describe a presheaf completion $\tilde{\mathbb I}$ of the theory $\mathbb I$ introduced above. We define this theory in such a way that every finitely generated model is finitely presented as a model of it.

We shall say that a term $t$ is in a given context $\vec{x}$, and we shall write $t(\vec{x})$, if all the variables occurring in it appear in the context.

\begin{lemma}\label{lem_zeroterms}
Any term $t(\vec{x})$ with source a tuple of variables $\vec{x}$ none of which is of sort $(X, \emptyset, 0)$ or $(X, \textrm{Spec}(k), 0)$ and whose target is $0$ is provably equivalent to $0$ in $\mathbb I$. 
\end{lemma}

\begin{proof}
We shall prove that all the terms $t(\vec{x})$ with target $0$ whose variables are all of sorts $(X, Y, n)$ for some $n\geq 1$ are provably equivalent to $0$ by induction on $n$. Notice that it is not possible that all the variables in $t$ are of sorts $(X, Y, 0)$ with $Y\neq \textrm{Spec}(k)$ and $Y\neq \emptyset$, since there is an edge in Nori's diagram $(X,Y,0)\to (X, \textrm{Spec}(k), 0)$ only if $Y\subseteq \textrm{Spec}(k)$. The induction step works as follows: by the naturality of the border we can suppose any term $t$ with a variable of sort $(X, Y, n)$ for $n\geq 1$ to be of the form \ac sums of terms of the form $t' \circ \delta$' where $t'$ involves only variables of sorts $(X, Y, l)$ with $l\lt n$. For $n=1$, it clearly suffices to prove the thesis under the assumption that $Y\neq \emptyset$, since one can always map by functoriality $(X, \emptyset, n)$ to $(X, X, n)$. Now, if $Y\neq \emptyset$ then by border naturality we can suppose that $t=t'\circ \delta$, where $\delta$ is the border $(X, Y, 1)\to (Y, \emptyset, 0)$ and $t$ is the arrow $!_{Y}:(Y, \emptyset, 0)\to (\textrm{Spec}(k), \emptyset, 0)$. But $!_{Y}=!_{X}\circ u$, where $u$ is the functoriality map $(Y, \emptyset, 0)\to (X, \emptyset, 0)$, whence $t'\circ \delta=!_{X}\circ u \circ \delta$. Now, by the complex condition we have $u\circ \delta=0$, whence $t=0$, as required.
\end{proof}

We add to the signature of $\mathbb I$ a predicate $R_{S}(\vec{x})$, where $S$ is a set of terms in the context $\vec{x}$ (with arbitrary target), and the geometric disjunctive axioms ensuring that the interpretation of this predicate in any model is the following: $R_{S}(\vec{x})$ if and only if for every term $t\in S$, $t(\vec{x})=0$ and for every term $t\notin S$, $t(\vec{x})\neq 0$.

We also add the \emph{Diers predicates} (in the sense of \cite{Diers}), that is the predicates $T_{n}(x_{1}, \ldots, x_{n})$ (for $n\geq 1$) whose meaning is the following: $T_{n}(x_{1}, \ldots, x_{n})$ if and only if $x_{n}$ is transcendental over the field ${\mathbb Q}(x_{1}, \ldots, x_{n-1})$. To ensure this meaning, we add the axioms specified in section 8.2 of \cite{OCPT}.

Notice that it is not unnatural to add these predicates and axioms to the theory $\mathbb I$ in order to obtain a theory of presheaf type. Indeed, the properties defined by these predicates are preserved by homomorphisms and filtered colimits of $\mathbb I$-models whence they are definable by a geometric formula in every presheaf completion of $\mathbb I$ by the definability theorem for theories of presheaf type (cf. Theorem 2.2 \cite{OC6}). 

Let us denote by $\mathbb J$ the theory obtained in this way.

Let us show that every finitely generated $\mathbb I$-model $(A, \vec{b})$ (where $A$ is the model and $\vec{b}$ are the generators) is finitely presented as a $\mathbb J$-model. For each context $\vec{x}$, let us define $S_{(A, \vec{b})}$ as the set of terms $t(\vec{x})$, where $\vec{x}$ is the context corresponding to the generators $\vec{b}$, such that $t(\vec{b})=0$ in $A$. 

By Lemma \ref{lem_zeroterms}, the field of coefficients $A_{0}$ of $(A, \vec{b})$ is finitely generated over $\mathbb Q$. Notice that, since $A$ is generated by $\vec{b}$, the elements of $A_{0}$ are precisely those of the form $t(\vec{b})$ for a term $t$ in the context $\vec{x}$ with target $0$. More specifically, we can suppose the context $\vec{x}$ to consist only of variables of sort $(X, \emptyset, 0)$ for some $X$, and the term $t$ to be a field combination of terms of the form $!_{X}(x)$ for some $X$. The field $A_{0}$ is therefore generated by the elements of the form $!_{X}(u)$ for some $u$ in $\vec{b}$.

In order to intrinsically characterize the field $A_{0}$ as a field extension of $\mathbb Q$ one has to know how the elements of the form $!_{X}(u)$ stand in relationship to one another in terms of algebraic (in)dependence relations. The Diers predicates serve for this purpose. Let $U_{(A, \vec{b})}$ be a formula in the context $\vec{x}$ written by using Diers' predicates which expresses these relations (there are clearly different such formulae $U_{(A, \vec{b})}$, but all of them are isomorphic in the syntactic category since they present the same model). Then the model $(A, \vec{b})$ is presented by the formula 
\[
\{\vec{x}. R_{S_{(A, \vec{b})}}(\vec{x}) \wedge U_{(A, \vec{b})}(\vec{x}) \}.
\]

For each non-empty context $\vec{x}$, let us denote by ${\cal U}_{\vec{x}}$ the collection of sets of the form $S_{(A, \vec{b})}$ for $(A, \vec{b})$ a finitely generated $\mathbb I$-model and $\vec{b}$ a tuple of type $\vec{x}$. Let us add to the theory $\mathbb J$ the following axioms:
\[
(\top \vdash_{\vec{x}}   \mathbin{\mathop{\textrm{ $\bigvee$}}\limits_{S\in {\cal U}_{\vec{x}}}} R_{S}(\vec{x})).
\]

By the results in section 7.2 of \cite{OCPT}, the resulting theory $\tilde{\mathbb I}$ is of presheaf type, and a presheaf completion of the theory $\mathbb I$.

The problem with this axiomatization of $\tilde{\mathbb I}$ is that it is not intrinsic, in the sense that the sets ${\cal U}_{\vec{x}}$ appearing in the axioms are defined in terms of the finitely generated $\mathbb I$-models. We would like to find an intrinsic characterization for them, or, in other words, to intrinsically characterize the $\tilde{\mathbb I}$-irreducible formulae, i.e. the formulae which present a $\tilde{\mathbb I}$-model. To this end, we investigate the problem of reconstructing a finitely generated $\mathbb I$-model from the corresponding formula presenting it. Given a set of terms $S$ in the context $\vec{x}=(x_{1}^{c_{1}}, \ldots, x_{m}^{c_{m}})$ (with arbitrary target), for whatever formula $U$ of the form $G(!_{X_{1}}(x_{i_{1}}) \slash z_{1}, \ldots, !_{X_{n}}(x_{i_{n}}) \slash z_{n})$ where the $x_{i_{j}}$ are precisely the variables among the ones in $\vec{x}$ whose sort is of the form $(X, \emptyset, 0)$ for some $X$, and the sort of $x_{i_{j}}$ is $(X_{i}, \emptyset, 0)$ for each $j=1, \ldots, n$ and $G(z_{1},\ldots, z_{n})$ is a formula which expresses the relations of algebraic (in)dependence of the elements $z_{1}, \ldots, z_{n}$ with respect to one another, we can try to build a (finitely generated) $\mathbb I$-model presented by the formula $\{\vec{x}. R_{S} \wedge U \}$. We have to impose some conditions on $S$ for it to be of the form $S_{(A, \vec{b})}$ for some $(A, \vec{b})$. For instance, any $S$ of the form $S_{(A, \vec{b})}$ enjoys the property that the sequent
 \[
(R_{S}(\vec{x}) \wedge U(\vec{x}) \vdash_{\vec{x}} \bot)
\]   
is \emph{not} provable in $\tilde{{\mathbb I}}$; indeed, it is satisfied by $(A, \vec{b})$. Moreover, $S_{(A, \vec{b})}$ is closed under the $A_{0}$-vector spaces operations, and satisfies the additional closure property: if the sequent 
\[
(R_{S}(\vec{x}) \wedge U(\vec{x}) \vdash_{\vec{x}} s(\vec{x})=0)
\]   
is provable in $\tilde{{\mathbb I}}$ then $s\in S$. In fact, this property follows from the previous one. Indeed, if $s\notin S$ then by definition of $R_{S}$, the sequent   
\[
(R_{S}(\vec{x}) \wedge U(\vec{x}) \vdash_{\vec{x}} s(\vec{x})\neq 0)
\]
is provable in $\tilde{{\mathbb I}}$ and hence the fact that 
\[
(R_{S}(\vec{x}) \wedge U(\vec{x}) \vdash_{\vec{x}} s(\vec{x})=0)
\]
is provable in $\tilde{{\mathbb I}}$ implies that the sequent
\[
(R_{S}(\vec{x}) \wedge U(\vec{x}) \vdash_{\vec{x}} \bot)
\] 
is provable in $\tilde{{\mathbb I}}$. 

The following proposition characterizes in intrinsic terms the formulae which present the $\mathbb I$-models in the theory $\tilde{{\mathbb I}}$.

\begin{proposition}\label{prop_irreducibleformulae}
Given a set $S$ of terms in the context $\vec{x}=(x_{1}^{c_{1}}, \ldots, x_{m}^{c_{m}})$ (with arbitrary target) and a formula $U=G(!_{X_{1}}(x_{i_{1}}) \slash z_{1}, \ldots, !_{X_{n}}(x_{i_{n}}) \slash z_{n})$ as specified above, the formula $\{\vec{x}.R_{S}(\vec{x}) \wedge U(\vec{x})\}$ presents a $\tilde{\mathbb I}$-model if and only if any of the following two equivalent conditions is satisfied:
\begin{enumerate}[(i)]
\item The sequent
\[
(R_{S}(\vec{x}) \wedge U(\vec{x}) \vdash_{\vec{x}} \bot)
\] 
is not provable in $\tilde{{\mathbb I}}$;

\item For any sequent
\[
(U(\vec{x}) \wedge w_{1}(\vec{x})=0\wedge \cdots \wedge w_{n}(\vec{x})=0 \vdash_{x} w(\vec{x})=0)
\] 
which is provable in $\tilde{\mathbb I}$, if $w_{1}, \ldots, w_{n}\in S$ then $w\in S$.
\end{enumerate}
\end{proposition}

\begin{proof}
Given $(S, U)$ satisfying the hypotheses of the proposition and condition $(i)$, consider the structure $A_{(S, U)}$ defined as follows. 

We set $A_{(S, U)}(0)$ equal to the field extension $K_{0}$ of $\mathbb Q$ with elements $\vec{\alpha}=(\alpha_{1}, \ldots, \alpha_{n})$ satisfying the formula $G$ from which the formula $U$ is built as $G(!_{X_{1}}(x_{i_{1}}) \slash z_{1}, \ldots, !_{X_{n}}(x_{i_{n}}) \slash z_{n})$.

For each $c\neq 0$,  we set 
\[
A_{(S, U)}(c)= (\mathbin{\mathop{\textrm{ $\bigoplus$}}\limits_{t:(c_{1}, \ldots, c_{m})\to c}} K_{0}1_{t})\slash E(c),
\]
where $E(c)$ is the subspace consisting of all the elements $k_{t_{1}}1_{t_{1}}+ \cdots +k_{t_m}1_{t_{m}}$ such that there exist rational functions $F_{1}, \ldots, F_{m}$ such that $k_{t_{i}}=F_{i}(\vec{\alpha})$ for all $i$ and the sequent
\[
(R_{S}(\vec{x}) \wedge U(\vec{x}) \vdash_{\vec{x}} F_{1}(s_{1}(\vec{x}), \ldots, s_{n}(\vec{x})) t_{1}(\vec{x})+ \cdots+ F_{m}(s_{1}(\vec{x}), \ldots, s_{n}(\vec{x})) t_{m}(\vec{x})=0)
\]  
is provable in $\tilde{\mathbb I}$, where the terms $s_{1},\ldots, s_{n}$ are precisely those of the form $!_{X}(x)$ for $x$ a variable inside the context $\vec{x}$ of sort  $(X, \emptyset, 0)$ for some $X$.

Notice in passing that 
\[
\mathbin{\mathop{\textrm{ $\bigoplus$}}\limits_{t:(c_{1}, \ldots, c_{m})\to c}} K_{0}1_{t}= \mathbin{\mathop{\textrm{ $\bigoplus$}}\limits_{t_{1}:c_{1}\to c}} K_{0}1_{t_{1}} \oplus \cdots \oplus  \mathbin{\mathop{\textrm{ $\bigoplus$}}\limits_{t:c_{m}\to c}} K_{0}1_{t_{m}}
\]
(recall that we take terms-in-context everywhere).

The definition of $A_{(S, U)}$ on edges of Nori's diagram with target different from $0$ is straightforward. The only non-obvious definition that we have to give is that for an edge $h:c\to 0$. For this, it clearly suffices to specify where the generators $1_{t}$ go. For any $t$, the term $h\circ t$ has target $0$, so it is a field combination $F(\ldots, !_{X}(x), \ldots)$ of terms of the form $!_{X}(x)$ for some $X$ and $x$ in $\vec{x}$. We stipulate that $1_{t}$ is sent to the element $F(\vec{\alpha})$ of $K_{0}$. We have to check that this is well-defined, i.e. that for any rational function $F'$ which is equivalent in the theory of (Diers) fields to $F$, $F(\vec{\alpha})=F'(\vec{\alpha})$. But this follows at once from the fact that $K_{0}$ is a (Diers) field.

Let us now prove that $A_{(S, U)}$ is a model of $\tilde{{\mathbb I}}$ presented by the formula $\{\vec{x}. R_{S} \wedge U \}$. First, we have to check that it is a model of $\tilde{{\mathbb I}}$, equivalently of $\mathbb I$ (notice that the category of $\mathbb I$-models is equivalent to the ind-completion of the category of finitely generated $\mathbb I$-models). To this end, it is essential to establish the independence from the rational functions $F_{i}$ in the definition of the subspace $E(c)$. This follows from the fact that, by definition of $K_{0}$, for any $F$ and $G$ such that $F(\vec{\alpha})=G(\vec{\alpha})$ in $K_{0}$, the sequent
\[
(U(\vec{x}) \vdash_{\vec{x}} F(s_{1}(\vec{x}), \ldots, s_{n}(\vec{x}))=G(s_{1}(\vec{x}), \ldots, s_{n}(\vec{x})))
\]
is provable in $\tilde{\mathbb I}$. 

Therefore, we have that $F_{1}(\vec{\alpha})1_{t_{1}}+ \cdots +F_{m}(\vec{\alpha})1_{t_{m}}\in E(c)$ \emph{if and only if} the sequent
\[
(R_{S}(\vec{x}) \wedge U(\vec{x}) \vdash_{\vec{x}} F_{1}(s_{1}(\vec{x}), \ldots, s_{n}(\vec{x})) t_{1}(\vec{x})+ \cdots+ F_{m}(s_{1}(\vec{x}), \ldots, s_{n}(\vec{x})) t_{m}(\vec{x})=0)
\]  
is provable in $\tilde{\mathbb I}$.

For an element $u=F_{1}(\vec{\alpha})1_{t_{1}}+ \cdots +F_{m}(\vec{\alpha})1_{t_{m}}$ of $\mathbin{\mathop{\textrm{ $\bigoplus$}}\limits_{t:(c_{1}, \ldots, c_{m})\to c}} K_{0}1_{t}$, let us denote by $\tilde{u}_{\vec{x}}$ the term $F_{1}(s_{1}(\vec{x}), \ldots, s_{n}(\vec{x})) t_{1}(\vec{x})+ \cdots+ F_{m}(s_{1}(\vec{x}), \ldots, s_{n}(\vec{x})) t_{m}(\vec{x})$. It is immediate to see, by using the characterization of the elements of $E(c)$ just established, that the assignment $u \to \tilde{u}_{\vec{x}}$ respects the term construction up to provable equivalence in $\tilde{\mathbb I}$. This clearly implies that $A_{(S, U)}$ is a model of $\mathbb I$. 

To conclude the proof that $A_{(S, U)}$ is presented by the formula $\{\vec{x}. R_{S}(\vec{x}) \wedge U(\vec{x})\}$, it clearly suffices to check that the generators satisfy the predicate $R_{S}$. This amounts to verifying that for any term $t$, $t\in E(c)$ if and only if $t\in S$. The \ac if' direction is trivial, and the \ac only if' one follows from condition (i) of the proposition. Using this remark, it is immediate to see that one can recover the formula from the model presented by it in such a way that the assignments $(A, \vec{b})\to \{\vec{x}. R_{S_{(A, \vec{b})}} \wedge U_{(A, \vec{b})} \}$ and $(S, U)\to A_{(S, U)}$ are inverses to each other (up to isomorphism) under the hypothesis that $(S, U)$ satisfies the conditions of the proposition. 

It remains to check that conditions (i) and (ii) of the Proposition are equivalent. The fact that (i) implies (ii) follows immediately from the definition of $R_{S}$. Conversely, let us assume condition (ii) and deduce condition (i). Suppose that the sequent 
\[
(R_{S}(\vec{x}) \wedge U(\vec{x}) \vdash_{\vec{x}} \bot)
\] 
is provable in $\tilde{{\mathbb I}}$. By recalling the definition of $R_{S}$ and $U$ and the fact that in classical logic a sequent $(A \wedge \neg B \vdash C)$ is equivalent to the sequent $(A \vdash B\vee C)$, we obtain that the sequent
\[
(R_{S}(\vec{x}) \wedge U(\vec{x}) \vdash_{\vec{x}} \bot)
\]
is equivalent to a sequent whose premise is a finite conjunction of polynomial conditions and whose conclusion is an infinitary disjunction of conditions of the form $t(\vec{x})=0$ or $t(\vec{x})\neq 0$ or polynomial conditions. Now, this sequent is written over the signature of the theory $\mathbb I$, which is coherent, and has as premise a coherent formula. Since $\mathbb I$ and $\tilde{\mathbb I}$ have the same set-based models, this sequent is valid in all set-based $\mathbb I$-models and hence, assuming the axiom of choice, it is provable in $\mathbb I$. But $\mathbb I$ is coherent, so the provability of this sequent entails that of a sequent having the same premise and as consequence a finite subdisjunction of the disjunction appearing in the conclusion of the former sequent (cf. Theorem 3.5 \cite{OC7}). Therefore a sequent whose premise is a finite conjunction of formulae of the form $P(\vec{x})=0$ (where $P$ is a polynomial) or $t(\vec{x})=0$, for $t\in S$, and whose conclusion is a finite disjunction of formulae of the form $P(\vec{x})=0$ or $s(\vec{x})=0$, where $s\notin S$, is provable in $\mathbb I$. But this sequent is written over the signature of the cartesianization of $\mathbb I$, whence the disjunction in the conclusion can be supposed to be a singleton since the formula in the premise is cartesian. Therefore a sequent of the form $(U(\vec{x}) \wedge t_{1}(\vec{x}) \wedge \cdots \wedge t_{n}(\vec{x}) \vdash_{\vec{x}} s(\vec{x})=0)$ is provable in $\tilde{\mathbb I}$, where $t_{1}, \ldots, t_{n}\in S$ and $s\notin S$. But our hypothesis implies that $s\in S$, which gives a contradiction.
\end{proof}

\subsection{The homomorphisms of finitely generated $\mathbb I$-models}\label{sec_homomorphisms}

In this section, we shall obtain a characterization of the homomorphisms of finitely generated $\mathbb I$-models based on the syntactic description of the formulae which present them provided by Proposition \ref{prop_irreducibleformulae}. Given two finitely generated $\mathbb I$-models $(A, \vec{b})$ and $(A', \vec{b}')$, if $(A', \vec{b}')$ (resp. $(A, \vec{b})$ ) is presented by a formula $\{\vec{y}. R_{T}(\vec{y}) \wedge V(\vec{y})\}$ (resp. by a formula $\{\vec{x}. R_{S}(\vec{x}) \wedge U(\vec{x})\}$) then the $\tilde{\mathbb I}$-model homomorphisms $(A', \vec{b}')\to (A, \vec{b})$ correspond bijectively to the (equivalence classes of) $\tilde{\mathbb I}$-provably functional formulae $\theta(\vec{x}, \vec{y}):\{\vec{x}. R_{S}(\vec{x}) \wedge U(\vec{x})\} \to \{\vec{y}. R_{T}(\vec{y}) \wedge V(\vec{y})\}$, via the assignment sending $\theta$ to the homomorphism sending the generators $\vec{b'}$ to the tuple $[[\theta]]_{(A, \vec{b})}(\vec{b})$. Since $(A, \vec{b})$ is finitely generated, we can suppose without loss of generality $\theta$ to be of the form $\vec{y}=\vec{w}(\vec{x})$, where $\vec{w}$ is a tuple of terms in the context $\vec{x}$.  

\begin{proposition}\label{prop_arrows}
With the above notation, given an arrow $\vec{y}=\vec{w}(\vec{x}):\{\vec{x}. R_{S}(\vec{x}) \wedge U(\vec{x})\} \to \{\vec{y}. R_{T}(\vec{y}) \wedge V(\vec{y})\}$ in the syntactic category of the theory $\tilde{\mathbb I}$, $(T, V)$ is uniquely determined by $(S, U)$ and $\vec{w}$, as follows: 
\begin{enumerate}[(i)]
\item $T=\{t \mid \textrm{the sequent } (R_{S}(\vec{x}) \wedge U(\vec{x}) \vdash_{\vec{x}} t(\vec{w}(\vec{x}))=0) \textrm{ is provable in } \tilde{{\mathbb I}} \}=\{t \mid t\circ \vec{w}\in S\}$;

\item $V(\vec{y})$ is the (finite) conjunction of the Diers conditions $D(\vec{y})$ such that the sequent $(R_{S}(\vec{x}) \wedge U(\vec{x}) \vdash_{\vec{x}} D(\vec{w}(\vec{x})))$ is provable in $\tilde{{\mathbb I}}$.  
\end{enumerate} 
\end{proposition}

\begin{proof}
(i) Let us set 
\[
T_{0}:=\{t \mid \textrm{ the sequent } (R_{S}(\vec{x}) \wedge U(\vec{x}) \vdash_{\vec{x}} t(\vec{w}(\vec{x}))=0) \textrm{ is provable in } \tilde{{\mathbb I}} \}.
\]
We have to prove that $T=T_{0}$. Let us start by proving that $T_{0}\subseteq T$. Notice that, since $\vec{y}=\vec{w}(\vec{x})$ is an arrow $\{\vec{x}. R_{S}(\vec{x}) \wedge U(\vec{x})\} \to \{\vec{y}. R_{T}(\vec{y}) \wedge V(\vec{y})\}$ in the syntactic category of the theory $\tilde{\mathbb I}$, the sequent
\[
(R_{S}(\vec{x}) \wedge U(\vec{x}) \vdash_{\vec{x}}  R_{T}(\vec{w}(\vec{x})) \wedge V(\vec{w}(\vec{x}))  )
\]
is provable in $\tilde{{\mathbb I}}$.

If $t\in T_{0}$ then $t\in T$; indeed, if we had $t\notin T$ then the sequent 
\[
(R_{S}(\vec{x}) \wedge U(\vec{x}) \vdash_{\vec{x}} \bot)
\]
would be provable in $\tilde{{\mathbb I}}$, which is contrary to our hypotheses. This proves that $T_{0}\subseteq T$. Let us now prove the converse inclusion. Given a term $t$, the $\tilde{{\mathbb I}}$-irreducibility of the formula $\{\vec{x}.R_{S}(\vec{x}) \wedge U(\vec{x})\}$ implies that either $t\in T_{0}$ or the sequent 
\[
\sigma:=(R_{S}(\vec{x}) \wedge U(\vec{x}) \vdash_{\vec{x}}  t(\vec{w}(\vec{x}))\neq 0)
\]
is provable in $\tilde{{\mathbb I}}$. But the latter possibility is incompatible with the condition $t\in T$, since the provability of $\sigma$ implies that of the sequent $(R_{S}(\vec{x}) \wedge U(\vec{x}) \vdash_{\vec{x}} \bot$, which is contrary to our hypotheses. 

(ii) This follows by noticing that all the Diers predicates $D$ are both preserved and reflected by homomorphisms of $\tilde{{\mathbb I}}$-models, in view of Lemma \ref{lem_fp}.
\end{proof}

\begin{remark}
The proposition shows that instead of specifying an arrow in the syntactic category of the theory $\tilde{{\mathbb I}}$, it suffices to specify its domain $\{\vec{x}. R_{S}(\vec{x}) \wedge U(\vec{x})\}$ and a tuple of terms $\vec{w}(\vec{x})$ in the context $\vec{x}$. Indeed, the pair $(T, V)$ defined in the statement of the proposition satisfies the hypotheses of Proposition \ref{prop_irreducibleformulae}, whence $\vec{y}=\vec{w}(\vec{x})$ defines an arrow $\{\vec{x}. R_{S}(\vec{x}) \wedge U(\vec{x})\} \to \{\vec{y}. R_{T}(\vec{y}) \wedge V(\vec{y})\}$ between $\tilde{\mathbb I}$-irreducible formulae in the syntactic category of the theory $\tilde{\mathbb I}$. 
\end{remark}

The $T$ (resp. $V$) defined in the statement of the proposition will be denoted respectively by $\vec{w}_{\ast}(S)$ and by $\vec{w}_{\ast}(U)$.

\subsection{The syntactic approach to the amalgamation property}\label{sec_amalgamation}

In this section we reformulate the amalgamation property on the category of finitely generated $\mathbb I$-models in syntactic terms, using the characterization provided by Proposition \ref{prop_irreducibleformulae}.

Since every formula can be covered by irreducible formulae, the amalgamation property can be formulated at the syntactic level by requiring that the pullback of two arrows between irreducible formulae be non-zero (i.e., a non-contradictory formula) in the syntactic category of $\tilde{{\mathbb I}}$. 

Given $(S, U)$ and $(T, V)$, such that the formulae $\{\vec{x}. R_{S}(\vec{x}) \wedge U(\vec{x})\}$ and $\{\vec{y}. R_{T}(\vec{y}) \wedge V(\vec{y})\}$ are not $\tilde{\mathbb I}$-provably contradictory, and two terms $\vec{w}(\vec{x})$ and $\vec{z}(\vec{y})$ with the same target such that $\{t \mid t\circ \vec{w} \in S \}=\{t \mid t\circ \vec{z} \in S \}$ and for any Diers predicate $D$,  $(R_{S}(\vec{x}) \wedge U(\vec{x}) \vdash_{\vec{x}} D(\vec{w}(\vec{x})))$ is provable in $\tilde{{\mathbb I}}$ if and only if $(R_{T}(\vec{y}) \wedge V(\vec{y}) \vdash_{\vec{y}} D(\vec{z}(\vec{y})))$ is provable in $\tilde{{\mathbb I}}$, the condition amounts to requiring that the formula  $\{\vec{x}, \vec{y}. R_{S}(\vec{x}) \wedge U(\vec{x}) \wedge R_{T}(\vec{y}) \wedge V(\vec{y}) \wedge \vec{w}(\vec{x})=\vec{z}(\vec{y}) \}$ should not be contradictory in $\tilde{{\mathbb I}}$, i.e. that the sequent 
\[
(R_{S}(\vec{x}) \wedge U(\vec{x}) \wedge R_{T}(\vec{y}) \wedge V(\vec{y}) \wedge \vec{w}(\vec{x})=\vec{z}(\vec{y}) \vdash_{\vec{x}, \vec{y}} \bot)
\]
should not be provable in $\tilde{{\mathbb I}}$.

We shall now prove that this is indeed the case, i.e. that the amalgamation property holds in the category of finitely generated $\mathbb I$-models.

The sequent
\[
(R_{S}(\vec{x}) \wedge U(\vec{x}) \wedge R_{T}(\vec{y}) \wedge V(\vec{y}) \wedge \vec{w}(\vec{x})=\vec{z}(\vec{y}) \vdash_{\vec{x}, \vec{y}} \bot)
\]
is provable in $\tilde{{\mathbb I}}$ if and only if the sequent 
\[
(R_{S}(\vec{x}) \wedge U(\vec{x}) \wedge R_{T}(\vec{y}) \wedge V(\vec{y}) \vdash_{\vec{x}, \vec{y}} \vec{w}(\vec{x})\neq \vec{z}(\vec{y})   )
\]
is provable in $\tilde{{\mathbb I}}$.

Now, given the form of the axioms of $\mathbb I$, the only way to deduce an inequality in the conclusion of a sequent from a set of inequalities in its premise is by means of a term $t$ such that the premise of the sequent entails the equality $(t\circ \vec{w})(\vec{x})=(t\circ \vec{z})(\vec{y})$ and either $(t\circ \vec{w})(\vec{x})=0$ (i.e., $t\circ \vec{w}\in S$ if $t$ has target $\neq 0$) and $(t\circ \vec{z})(\vec{y})\neq 0$ (i.e., $t\circ \vec{z}\notin T$ if $t$ has target $\neq 0$)  or $(t\circ \vec{w})(\vec{x})\neq 0$ (i.e., $t\circ \vec{w}\notin S$ if $t$ has target $\neq 0$) and $(t\circ \vec{z})(\vec{y})=0$ (i.e., $t\circ \vec{z}\in T$ if $t$ has target $\neq 0$). Now, depending on the fact that $t$ has target $0$ or $\neq 0$, our hypotheses ensure that none of these possibilities can arise. This completes our proof. 

\section{The theory of homogeneous models}\label{sec:theoryhomogeneousmodels}

Thanks to the characterization of the homomorphisms between finitely generated $\mathbb I$-models obtained in section \ref{sec_homomorphisms}, we can obtain an explicit axiomatization of the theory of homogeneous $\tilde{\mathbb I}$-models.

\begin{theorem}\label{thm_axiomatizationhomogeneous}
The theory of homogeneous $\tilde{\mathbb I}$-models is obtained from $\tilde{\mathbb I}$ by adding all the sequents of the form 
\[
( R_{\vec{w}_{\ast}(S)}(\vec{y}) \wedge  \vec{w}_{\ast}(U)(\vec{y})  \vdash_{\vec{y}} (\exists \vec{x}) (R_{S}(\vec{x}) \wedge U(\vec{x}) \wedge \vec{y}=\vec{w}(\vec{x}))).
\]
for each $(S, U)$ such that the sequent 
\[
(R_{S}(\vec{x}) \wedge U(\vec{x}) \vdash_{\vec{x}} \bot)
\]
is not provable in $\tilde{{\mathbb I}}$ and any tuple of terms $\vec{w}(\vec{x})$ with target $\vec{y}$.
\end{theorem}\qed

By using Lemma \ref{lem_red}, we can semantically reformulate the homogeneity condition for a model $H$, as follows. For any finitely generated submodel $(A, \vec{\xi})$ of $H$ and any term $\vec{w}(\vec{\xi})$, every element $\vec{b}$ of $H$ such that $U_{\vec{b}}=U_{(A, \vec{w}(\vec{\xi}))}$ and $S_{\vec{b}}=S_{(A, \vec{w}(\vec{\xi}))}$ is of the form $\vec{w}(\vec{c})$ for some $\vec{c}$ such that $U_{(A, \vec{\xi})}=U_{\vec{c}}$ and $S_{(A, \vec{\xi})}=S_{\vec{c}}$ (equivalently, there exists an embedding $j:A\to H$ such that $\vec{b}=\vec{w}(j(\vec{\xi}))$).  

\begin{remark}
If for every term $\vec{w}$ there exists a term $\vec{g}$ such that the composite $\vec{g}\circ \vec{w}$ is provably equal to $0$ in the theory $\tilde{\mathbb I}$ then every homogeneity condition can be interpreted as a generalized exactness condition, entailing in particular the strong exactness of Theorem \ref{thm_exacthom}.
\end{remark}

\subsection{Global and sortwise homogeneity}

We notice that every ${\mathbb I}$-model which is sortwise finitely generated is finitely presentable; in particular, if it is $(0)$ in every place except for a finite number and finitely generated at every other sort then it is finitely presentable (equivalently, finitely generated).

From the existence of the models ${\cal I}_{K}$ considered in Remark \ref{rem:sortwisemodels} it follows that for every homogeneous ${\mathbb I}$-model $H$, $H(\textrm{Spec}(k), \emptyset, 0)$ is homogeneous as a field. It is immediate to see that a field of characteristic $0$ is homogeneous if and only if it is algebraically closed and has infinite transcendence degree (the last condition meaning constructively that for every natural number $n$, every set of $n$ algebraically independent elements over $\mathbb Q$ can be extended to a set of $n+1$ algebraically independent elements over $\mathbb Q$). Examples of homogeneous fields are the field $\mathbb C$ of complex numbers and the algebraic closures $\overline{{\mathbb Q}_{\ell}}$ of the fields ${\mathbb Q}_{\ell}$ of $\ell$-adic integers. It also follows immediately from the existence of the models ${\cal I}_{K}$ that for every finitely generated model $H$ of $\mathbb I$, the field of coefficients $H(\textrm{Spec}(k), \emptyset, 0)$ is finitely generated (recall the categorical characterization of finite generation of a model as the requirement that the covariant Hom functor on the category of models and injective homomorphisms should preserve filtered colimits). The finite generation of $H(\textrm{Spec}(k), \emptyset, 0)$ also follows as an immediate consequence of Lemma \ref{lem_zeroterms}. 

From this remark it follows that if we want to regard the different homology theories as homogeneous models of the theory ${\mathbb I}$, then we should associate to each object $(X, Y, i)$ of Nori's diagram not the singular homology groups with coefficients in $\mathbb Q$ (which are finite-dimensional vector spaces over $\mathbb Q$) or the $\ell$-adic homology groups (which are vector spaces over ${\mathbb Q}_{\ell}$) but their tensorizations with any homogeneous field of characteristic $0$, for instance $\mathbb C$ in the case of singular homology (notice that this amounts precisely to taking singular homology with coefficients in $\mathbb C$) or $\overline{{\mathbb Q}_{\ell}}$ in the case of $\ell$-adic homology. This does not constitute a problem in relation to our goal of proving the independence from $\ell$ of the dimension of the $\ell$-adic cohomology groups since the dimension over the field of coefficients is invariant with respect to the operation of extensions of scalars. 

It is natural to wonder if \ac global' homogeneity implies homogeneity for each sort, not just for that corresponding to the field of coefficients. To answer this question, we can try to fabricate models of $\mathbb I$ which are zero almost everywhere except for a given sort, as follows. Given a vertex of Nori's diagram $(X, Y, i)$, where $i\neq 0$, and a $K$-vector space $V$ (where $K$ is a field of characteristic $0$), we define the expansion $I_{(V, K)}$ of the model ${\cal I}_{K}$ defined in Remark \ref{rem:sortwisemodels} obtained by setting $I_{V}(X, Y, i)=V$ and $I_{V}(X', Y', i')=(0)$ for any other sort $(X', Y', i')$ where if $i'=0$ then $Y'=\emptyset$; we stipulate that ${\cal I}_{(V, K)}$ interprets any endomorphism of $(X, Y, i)$ as the identity and any other edge between different sorts as the zero arrow. 

If the structures ${\cal I}_{(V, K)}$ just defined are $\mathbb I$-models (we have not checked this but we suppose it to be the case) then for any $\mathbb I$-model $H$ and any sort $(X, Y, i)$ such that $X\neq Y$, the $H(\textrm{Spec}(k), \emptyset, 0)$-vector space $H(X, Y, i)$ is homogeneous as a model both of the injectivization of the theory of vector spaces over $\mathbb Q$ and as a model of the injectivization of the theory of vector spaces of a variable field of characteristic zero. Indeed, ${\cal I}_{(V, K)}$ is finitely generated as a $\mathbb I$-model if and only if $K$ is a finitely generated field and $V$ is of finite dimension over $K$. 

We shall now characterize the notion of homogeneity for these two theories.
 
In the injectivization ${\mathbb V}_{K}$ of the theory of vector spaces over a base field $K$ (which is of presheaf type, as remarked in section 8.6 of \cite{OCPT}), the homogeneous models are clearly the vector spaces which are infinite-dimensional over $K$.

\begin{proposition}
In the injectivization of the theory of vector spaces over a variable field, the homogeneous models are precisely the vector spaces $V$ over a field $K$ such that $K$ is a homogeneous field and for every finitely generated subfield $K'$ of $K$, $V$ is homogeneous as a $K'$-vector space (that is, as a model of the theory ${\mathbb V}_{K'}$ introduced above).
\end{proposition}

\begin{proof}
The fact that the condition is necessary is clear. We shall prove that it is also sufficient. Given embeddings $\xi=(\xi_{1}, \xi_{2}):(V_{0}, K_{0})\to (V, K)$ and $f=(f_{1}, f_{2}):(V_{0}, K_{0})\to (V_{1}, K_{1})$, we shall construct an embedding $\chi=(\chi_{1}, \chi_{2}):(V_{1}, K_{1})\to (V, K)$ such that $\chi \circ f=\xi$. Using the fact that $K$ is homogeneous as a field, we can find a field embedding $u:K_{1}\to K$ such that $u\circ f_{2}=\xi_{2}$. This allows us to consider $V$ as a $K_{1}$-vector space. Now, let $x_{1}, \ldots, x_{n}$ be a basis of $V_{0}$ as a $K_{0}$-vector space. The elements $f_{1}(x_{1}), \ldots, f_{1}(x_{n})$ might not be linearly independent anymore over $K_{1}$, but we can extract a maximal subset of elements $f_{1}(x_{i})$ (for $i\in S$) which are linearly independent over $K_{1}$. Consider the $K_{1}$-vector space $V'$ generated by these elements. We can define a $K_{1}$-vector space embedding into $(V, K)$ by sending each $f_{1}(x_{i})$, for $i\in S$, to $\xi_{1}(x_{i})$. This is well-defined since $f_{1}$ is injective by our hypothesis and the $f_{1}(x_{i})$, for $i\in S$, are linearly independent over $K_{1}$. On the other hand, the arrow $f$ factors through the canonical embedding $(V', K_{1})\hookrightarrow (V_{1}, K_{1})$; indeed, one can send each $x_{i}$ to the element $f_{1}(x_{i})\in V'$. Given this factorization of $f$, our thesis follows from the assumption that $V$ is homogeneous as a $K_{1}$-vector space.    
\end{proof}

\subsection{Homogeneity and exactness}\label{sec_homogex}

\begin{theorem}\label{thm_exacthom}
Every homogeneous $\tilde{\mathbb I}$-model $H$ satisfies the exactness conditions corresponding to the distinguished pairs in Nori's diagram. In fact, for any distinguished pair $(f,g)$ in Nori's diagram, the following strong exactness property holds: using the notation of section \ref{sec:presheafcompletion}, for every element $y$ in the kernel of $H(g)$ there exists an element $x$ such that $H(f)(x)=y$ and with the property that $H(s)(x)\neq 0$ for each term $s$ such that the equality $s(x)=0$ is not the conclusion of a sequent provable in the theory $\tilde{\mathbb I}$ whose premise is the conjunction of the Diers conditions of the form $D\circ f$, where $D$ is a Diers condition satisfied by $y$ in $H$, with conditions of the form $t(f(x))=0$ for terms $t\in S_{(H_{y}, y)}$.

In particular, there exists in the inverse image by $f$ of any element $y\in \textrm{Ker}(g)$ an element $x$ which is generic in the sense that for any term $s$ such that $H(s)(x)=0$, we have that $H(s)(x')=0$ for any $x'$ such that $H(f)(x')=y$.
\end{theorem}

\begin{proof}
We shall prove our thesis by using the syntactic characterization of the formulae which present the finitely generated $\mathbb I$-models in the theory $\tilde{{\mathbb I}}$ provided by Proposition \ref{prop_irreducibleformulae}. 

Given $(T, V)$ such that the sequent 
\[
(R_{T}(y) \wedge V(y) \vdash_{y} \bot)
\]
is not provable in $\tilde{\mathbb I}$ and a distinguished pair $(f, g)$ such that $g\in T$, we can define a pair $(S, U)$ such that $T=f_{\ast}(S)$ and $V=f_{\ast}(U)$ and such that the sequent
\[
(R_{S}(x) \wedge U(x) \vdash_{x} \bot)
\] 
is not provable in $\tilde{{\mathbb I}}$, as follows. 

We take $U$ as the finite conjunction of all the Diers conditions of the form $D\circ f$, where $D(y)$ is a Diers condition in $V$. 

We stipulate that $w\in S$ if and only if there exists a sequent
\[
(U(x) \wedge w_{1}(x)=0\wedge \cdots \wedge w_{n}(x)=0 \vdash_{x} w(x)=0)
\] 
provable in $\tilde{\mathbb I}$ where $w_{1}, \ldots, w_{n}$ are all terms of the form $t\circ f$ where $t\in T$.  Let us verify that $T=f_{\ast}(S)$. We clearly have $T\subseteq f_{\ast}(S)$. To prove the converse inclusion, we have to verify that if $t\circ f\in S$ then $t\in T$. If $t\circ f\in S$ then by definition of $S$ there exists a sequent 
\[
(U(x) \wedge w_{1}(x)=0\wedge \cdots \wedge w_{n}(x)=0 \vdash_{x} t(f(x))=0)
\]
which is provable in $\tilde{\mathbb I}$, where for each $i$ there exists $w_{i}'\in T$ such that $w_{i}=w_{i}' \circ f$.

By definition of $\tilde{\mathbb I}$ and of $U$, since $g\in T$, the sequent
\[
(V(y) \wedge g(y)=0 \wedge w_{1}'(y)=0 \wedge \cdots \wedge w_{n}'(y)=0 \vdash_{y} t(y)=0)
\]
is provable in $\tilde{\mathbb I}$, whence $t\in T$, as required.

It remains to prove that the sequent
\[
(R_{S}(x) \wedge U(x) \vdash_{x} \bot)
\] 
is not provable in $\tilde{{\mathbb I}}$. To show this, it suffices to observe that condition (ii) of Proposition \ref{prop_irreducibleformulae} is satisfied by the very definition of $(S, U)$.

Now, to deduce the strong exactness property for a homogeneous $\tilde{{\mathbb I}}$-model $H$, apply this to the pair $(T, V)$ given by: $T=S_{(H_{y}, y)}$ and $V=U_{(H_{y}, y)}$, where $H_{y}$ is the submodel of $H$ generated by the element $y$ (the notations being those of section \ref{sec:presheafcompletion}). Given an element $x$ such that $H(f)(x)=y$, the equality $H(s)(x)=0$ implies that $s(x)=0$ is the conclusion of a sequent provable in the theory $\tilde{\mathbb I}$ whose premise is the conjunction of $U(x)$ with conditions of the form $t(f(x))=0$ for terms $t\in T$, whence $H(s)(x')=0$ for any $x'$ for each $H(f)(x')=y$.
\end{proof}
 
\subsubsection{Extending finitely generated $\mathbb I$-models}

Given an element $b\in H(X', Y', i')$, consider the substructure $H_{b}$ of $H$ generated by $b$. Given the form of the axioms of $\mathbb I$, this is still a $\mathbb I$-model. 

Given a distinguished pair $(f, g)$ in Nori's quiver, where $(X, Y, i)$ and $(X', Y', i')$ are respectively the source and target of $f$, we wonder whether it is possible to extend $H_{b}$ by a $\mathbb I$-model $A$ generated by an element $\xi$ such that $A(f)(\xi)=b$. This will provide an alternative, semantic proof of the genericity property of homogeneous models with respect to distinguished pairs established in Theorem \ref{thm_exacthom}.   

The following lemma will be useful to us.

\begin{lemma}\label{lem_red}
Let $\tau$ be a geometric sequent over the signature of $\tilde{{\mathbb I}}$ whose conclusion is a disjunction of finite conjunctions of atomic formulae. If $\tau$ is provable in the theory obtained by adding to $\tilde{{\mathbb I}}$ all the axioms of the form 
\[
(g(x)=0 \vdash_{x} (\exists y)f(y)=x)
\]
for all distinguished pairs $(f, g)$ in Nori's diagram then $\tau$ is provable in $\tilde{{\mathbb I}}$.
\end{lemma}

\begin{proof}
Consider the classifying topos of the theory obtained by adding the set of sequents $\mathbb E$ specified in the lemma to the theory $\tilde{\mathbb I}$. By the duality theorem in \cite{OCL}, this topos can be represented as $\Sh({\cal C}, J)$, where $\cal C$ is the opposite of the category of finitely presentable $\tilde{\mathbb I}$-models and $J$ is generated by sieves generated by single arrows. By Proposition 2.6 \cite{OC7} and the fact that every geometric formula is provably equivalent in $\tilde{\mathbb I}$ to a disjunction of $\tilde{\mathbb I}$-irreducible formulae in the same context, we can suppose the premise and the conclusion of $\tau$ to be $\tilde{\mathbb I}$-irreducible formulae, i.e. formulae of the form $\{\vec{x}. R_{S}(\vec{x}) \wedge U(\vec{x})\}$ (cf. section \ref{sec:presheafcompletion}).

Now, the additional predicates that we have added have the meaning of an infinitary conjunction of formulae of the form $t(\vec{x})=0$ or of the form $t(\vec{x})\neq 0$. Now, a sequent of the form $(\psi \wedge t(\vec{x})\neq 0 \vdash \chi)$ is provably equivalent to the sequent $(\psi \vdash \chi \vee t(\vec{x})=0)$.

At the cost of replacing a single sequent by an equivalent set of sequents, we can suppose without loss of generality the conclusion of $\tau$ to be a formula of the form $t(\vec{x})=0$ or of the form $t(\vec{x})\neq 0$ (recall that the predicates $R_{S}$ and $U$ have the same meaning as infinitary conjunctions of such conditions).

Let us first consider the second case, that is when the conclusion of $\tau$ is given by $t(\vec{x})\neq 0$. Taking the counterpositive of this sequent, we obtain a sequent whose premise is the formula $t(\vec{x})=0$ and whose conclusion is a (possibly infinitary) disjunction of formulae of the form $t'(\vec{x})=0$ or $t'(\vec{x})\neq 0$. Given the fact that $\tilde{{\mathbb I}}\cup {\mathbb E}$ has the same set-based models as the coherent theory ${\mathbb I}\cup {\mathbb E}$, it follows by assuming the axiom of choice that this sequent is provable in ${\mathbb I}\cup {\mathbb E}$ and hence that the disjunction on the right hand side can be supposed to be finitary without loss of generality (cf. Theorem 3.5 \cite{OC7}). But this sequent is equivalent to one whose premise is a finite conjunction of literals of the form $t(\vec{x})=0$ (a sequent of the form $(\psi \wedge t(\vec{x})=0 \vdash \chi)$ is provably equivalent to the sequent $(\psi \vdash \chi \vee t(\vec{x})\neq 0)$) and whose conclusion is a finite disjunction of literals of the form $t(\vec{x})=0$. This sequent is thus provable in the regular theory ${\mathbb S}_{\cal K}\cup {\mathbb E}$ (since this theory is finitary and has the same set-based models as the theory ${\mathbb I}\cup {\mathbb E}$) and the fact that the formula in the premise of the sequent is regular implies that one can suppose without loss of generality the finite family over which the disjunction on the right hand side of the sequent is indexed to be a singleton (cf. Theorem 3.4 \cite{OC7}). Therefore $\tau$ is an algebraic sequent whence it is provable in ${\mathbb I}$, as required.

Let us now turn to the first case, i.e. that of a sequent whose conclusion is a formula of the form $t(\vec{x})=0$. Taking the counterpositive of this sequent, we obtain a sequent whose premise is the formula $t(\vec{x})\neq 0$ and whose conclusion is a (possibly infinitary) disjunction of formulae of the form $t'(\vec{x})=0$ or $t'(\vec{x})\neq 0$. This sequent is clearly equivalent to one whose premise is $\top$ and whose conclusion is a disjunction of formulae of the form $t'(\vec{x})=0$ or $t'(\vec{x})\neq 0$. We can then conclude by the same arguments as above. 
\end{proof}

\begin{remark}
Lemma \ref{lem_red} shows that, as it can be naturally expected, completing the theory $\mathbb I$ to a theory of presheaf type preserves the equivalence between provability of geometric sequents whose conclusion is a disjunction of Horn formulae in the theory and in the quotient of it obtained by adding the sequents in $\mathbb E$.    
\end{remark}

\begin{theorem}\label{thm_constructionA}
Suppose that $H_{b}$ is an $\mathbb I$-model finitely generated by an element $b\in H_{b}(X, Y, i)$. Then for any distinguished pair $(f_{0}, g_{0})$ such that $g_{0}(b)=0$ there exists a $\mathbb I$-model $A$ in which $H_{b}$ embeds containing an element $a$ such that $A(f_{0})(a)=b$.
\end{theorem}

\begin{proof}
Let us denote by $K_{0}$ the (finitely generated) field of coefficients of the model $H_{b}$, and by $c_{0}:=(X, Y, i)$ and $c_{0}':=(X', Y', i')$ respectively the codomain and the domain of $f_{0}$.

We define $A$ as follows. We set $A(\textrm{Spec}(k), \emptyset, 0)=K_{0}$. Notice that, by definition of the structure generated by an element $b$, every element of $K_{0}$ is of the form $t(b)$ where $t(z)$ is a term of type $(X, Y, i)\to (\textrm{Spec}(k), \emptyset, 0)$; of course, $t$ is not uniquely determined since there can be different terms $t(z)$ and $t'(z)$ such that $t(b)=t'(b)$.  For any sort $c\neq 0$, we set $A(c)$ equal to the quotient of the direct sum $V(c):=\bigoplus_{s} K_{0}1_{s} \oplus H_{b}(c)$, where the big direct sum is indexed over all the terms $s:c_{0}'\to c$ by the subspace $E(c)$ defined as follows. 

We stipulate that an element $\alpha:=\sum_{s\in S}k_{s}1_{s}+u$ (where $S$ is a finite subset of the set of terms $c_{0}'\to c$) belongs to $E(c)$ if and only if there exists a term $G(w)$ of type $(X, Y, i)\to c$ such that $u=G(b)$ and for each $s\in S$ a term $F_{s}(z)$ of type $(X, Y, i)\to (\textrm{Spec}(k), \emptyset, 0)$ such that $k_{s}=F_{s}(b)$ and the sequent  
\[
(\phi_{H_{b}}(f_{0}(y)) \vdash_{y} \sum_{s\in S} F_{s}(f_{0}(y)) \vec{s}(y)+ G(f_{0}(y))=0) 
\]
is provable in $\tilde{{\mathbb I}}$, where $\phi_{H_{b}}(x)$ is a formula which is satisfied by $b$ in $H_{b}$ and which provably implies in $\tilde{{\mathbb I}}$ any formula $h(x)=0$ such that $h(b)=0$ (for instance, one can take $\phi_{H_{b}}$ to be the infinitary conjunction of all such formulae - this is equivalent to a geometric formula in ${\mathbb I}$ by the definability theorem for theories of presheaf type, cf. Theorem 2.2 \cite{OC6}). 

For example, we have that $1_{f_{0}}-b$ belongs to $E(c_{0})$.

It is important to remark the following fact: for \emph{any} other representation $\alpha:=\sum_{s'\in S'}F'_{s'}(b)1_{s'}+G'(b)$ of our element $\alpha$, if $\alpha$ belongs to $E(c)$ then the sequent
\[
(\phi_{H_{b}}(f_{0}(y)) \vdash_{y} \sum_{s'\in S'} F'_{s'}(f_{0}(y)) s'(y)+ G'(f_{0}(y))=0) 
\] 
is provable in $\tilde{\mathbb I}$. Indeed, the equality 
\[
\sum_{s'\in S'}F'_{s'}(b)1_{s'}+G'(b)=\sum_{s\in S}F_{s}(b)1_{s}+G(b)
\]
is equivalent to the conjunction of the two conditions 
\[
\sum_{s'\in S'}F'_{s'}(b)1_{s'}=\sum_{s\in S}F_{s}(b)1_{s}
\]
and 
\[
G(b)=G'(b).
\]
But if all the coefficients of the $1_{s}$ and $1_{s'}$ are non-zero, as we can clearly suppose without loss of generality, the former condition is equivalent to the conjunction of the conditions $S=S'$ and $F'_{s}(b)=F_{s}(b)$ for each $s\in S$. Now, since all the sequents
\[
(\phi_{H_{b}}(f_{0}(y)) \vdash_{y} F'_{s}(f_{0}(y))=F_{s}(f_{0}(y)))
\]
(for $s\in S$) and
\[
(\phi_{H_{b}}(f_{0}(y)) \vdash_{y} G(f_{0}(y))=G(f_{0}(y)))
\]
are provable in $\tilde{{\mathbb I}}$, our claim follows from the equality axioms for geometric logic.

We can thus conclude that an element $\sum_{s\in S}F_{s}(b)1_{s}+G(b)$ belongs to $E(c)$ \emph{if and only if} the sequent 
\[
(\phi_{H_{b}}(f_{0}(y)) \vdash_{y} \sum_{s\in S} F_{s}(f_{0}(y)) s(y)+ G(f_{0}(y))=0) 
\]
is provable in $\tilde{{\mathbb I}}$.  

Let us first verify that $A$ is a $\mathbb I$-model, and that $H_{b}$ embeds into $A$. First, we have to check that $E(c)$ is indeed a subspace and that the assignment $c\to E(c)$ is functorial. Given an element $\alpha=\sum_{s\in S}F_{s}(b)1_{s}+G(b)$, we denote by $\tilde{\alpha}_{y}$ the expression $\sum_{s\in S} F_{s}(f_{0}(y))s(y)+ G(f_{0}(y))$. We have already observed that under the premise $\phi_{H_{b}}(f_{0}(y))$ this is well-defined up to provable equality $\sim$ in $\tilde{{\mathbb I}}$, that is, it does not depend on the representation of $\alpha$. To conclude that $E(c)$ is a subspace it suffices to observe that the operation $\alpha \to \tilde{\alpha}_{y}$ respects the sum and product by a scalar (up to the relation $\sim$). So $A$ is well-defined and satisfies the functoriality and border naturality axioms. The fact that it satisfies the complex condition is also straightforward. 
  
The fact that $A$ satisfies the axioms for connected components will follow from the fact that $H_{b}$ does once we have proved that it embeds into $A$, since the field of coefficients of $H_{b}$ and $A$ is the same.

Let us prove that $H_{b}$ embeds into $A$. We have to prove that if $u\in E(c)$ with $u\in H_{b}(c)$ then $u=0$. 

If $u=G(b)\in E(c)$ then by definition of $E(c)$ the sequent 
\[
(\phi_{H_{b}}(f_{0}(y))) \vdash_{y} G(f_{0}(y)))=0) 
\]
is provable in $\tilde{{\mathbb I}}$. Therefore, by Lemma \ref{lem_red}, the sequent
\[
(g_{0}(x)=0 \wedge \phi_{H_{b}}(x) \vdash_{x} G(x)=0)
\]
is provable in $\tilde{{\mathbb I}}$. Now, since $H_{b}$ is a model of $\tilde{\mathbb I}$, it follows that $G(b)=0$, as required.

The fact that $A$ is a model of $\mathbb I$ can be proved by using the same method employed in the proof of Proposition \ref{prop_irreducibleformulae}.
\end{proof}

\begin{remarks}
\begin{enumerate}[(i)]
\item The model $A$ extending $H_{b}$ constructed in the above proof has the property that every $\mathbb I$-model $H'$ containing $H_{b}$ and an element $a$ such that $H'(f_{0})(a)=b$ is a quotient of $A$.

\item If the model $H_{b}$ is presented by a formula $\{x. R_{S}(x) \wedge U(x)\}$ then the model $A$ constructed in the above proof is presented by the formula $\{y. R_{T}(y) \wedge V(y)\}$ where $(T, V)$ is defined as in Proposition \ref{prop_arrows}. Indeed, by definition of $A$, $s(a)=0$ in $A$ if and only if the sequent
\[
(\phi_{H_{b}}(f_{0}(y))) \vdash_{y} s(y)=0) 
\]
is provable in $\tilde{\mathbb I}$ (recall that $a=1_{id}$ is the generator of $A$). 

Without loss of generality, one can suppose the formula $\phi_{H_{b}}$ to be the infinitary conjunction $C_{b}$ of all the formulae of the form $t(x)=0$ for $t$ a term (possibly with target $0$) such that $t(b)=0$ in $H_{b}$ (notice in particular that any $t\in T$ is of this form). Indeed, any formula $\phi_{H_{b}}$ which provably entails all the formulae $h(x)=0$ such that $h(b)=0$ and which is satisfied by $b$ in $H_{b}$ is provably entailed by $R_{S}(y) \wedge U(y)$ (by the universal property of $H_{b}$ as model presented by that formula) and it is not hard to see that a sequent 
\[
(R_{S}(x) \wedge U(x) \vdash_{x}  h(x)=0)
\]
is provable in $\tilde{\mathbb I}$ if and only if the sequent
\[
(C_{b}(x) \vdash_{x}  h(x)=0)
\]
is provable in $\tilde{\mathbb I}$.

Taking $\phi_{H_{b}}=C_{b}$, one sees immediately that the sequent 
\[
(\phi_{H_{b}}(f_{0}(y))) \vdash_{y} s(y)=0) 
\]
is provable in $\tilde{\mathbb I}$ if and only if $s\in T$ and similarly for the case of terms of target $0$.
\end{enumerate}
\end{remarks}

\begin{corollary}
Any $\mathbb I$-homogeneous model $H$ in which $H_{b}$ embeds satisfies the property that there exists an element $a_{0}\in H(\dom(f_{0}))$ such that any formula of the form $\sum_{s\in S} F_{s}(f_{0}(y))s(y)+ G(f_{0}(y))=0$ which is satisfied by $a_{0}$ is provable in $\mathbb I$. In particular, we have the following \emph{uniformity property} for the counterimages of the element $b$: if $\sum_{s\in S} H(F_{s})(b)H(s)(a_{0})+ H(G)(b)=0$ then for any other $a$ such that $H(f_{0})(a)=b$, we have that $\sum_{s\in S} H(F_{s})(b) H(s)(a)+ H(G)(b)=0$.  
\end{corollary}

\begin{proof}
This is an immediate consequence of the theorem obtained by explicitly unravelling what it means that there exists an injective homomorphism $A\to H$ extending the embedding $H_{b}\hookrightarrow H$.
\end{proof}

\section{A syntactic triangulated category}\label{sec:triangcat}

In this section we shall define a triangulated category directly built from the category of schemes by extracting the essential formal features which are common to all (known) (co)homology theories of schemes. We shall define this category by means of an inductive procedure involving an increasing sequence of theories. All our schemes will be over a fixed base $k$.

Let $L_{0}$ be the signature consisting of a sort $K_{0}$ (for formalizing the coefficients field), one sort for each pair $(X, n)$, where $n$ is a non-positive integer and $X$ a scheme over $k$ (we set $K_{0}=(\textrm{Spec}(k), 0)$), a function symbol $(f, n):(X, n)\to (Y, n)$ for each morphism of schemes $f:X\to Y$ over $k$, a function symbol $0:(X, n)\to (Y, m)$ for each pair of sorts (formalizing the zero arrows on homology groups) and the function symbols formalizing the structure of ring on $K_{0}$ and the structure of $K_{0}$-module on the sorts $(X, n)$. 

Let ${\mathbb T}_{0}$ be the theory over $L_{0}$ whose axioms are the functoriality axioms, the sequents defining the structure of ring on $K_{0}$ and of $K_{0}$-module on the other sorts, and the axioms ensuring that the all the function symbols $f$ are module homomorphisms. 

Let us denote by $\sim_{0}$ the equivalence relation on terms over $L_{0}$ given by: $t\sim t'$ if and only if the sequent $(\top \vdash t=t')$ is provable in ${\mathbb T}_{0}$. 

Let ${\cal C}_{0}$ be the algebraic syntactic category of the theory ${\mathbb T}_{0}$ (in which we only allow formulae of the form $\{\vec{x}. \top\}$ or of the form $\{\vec{x}.\vec{x}=0\}$). The arrows in such a category are given by equivalence classes of terms (in our case the equivalence relation being $\sim_{0}$).

We have a translation functor $T_{1}$ on the objects of ${\cal C}_{0}$ sending each $(X, n)$ to $(X, n-1)$; the set of distinguished triangles in ${\cal C}_{0}$ is set to consist just of the trivial ones (i.e., those of the form $(id: X \to X, 0:X\to 0)$).  

Notice that for any term $t:c_{1}, \ldots, c_{k}\to c$ over $L_{0}$ there is a translated term $T_{1}t:T_{1}c_{1}, \ldots, T_{1}c_{k}\to T_{1}c$. 

Let $L_{l+1}$ be the signature obtained from $L_{l}$ by adding a new sort $R_{t}$ for each ($\sim_{l}$-equivalence class of) term-in-context $t:c_{1}, \ldots, c_{m}\to c$ over $L_{l}$ which is not ${\mathbb T}_{l}$-provably equivalent to the identity and which does not belong to $L_{l-1}$, a function symbol $\pi_{t}:c\to R_{t}$ and function symbols $\delta^{i}_{t}:R_{t} \to T_{1}(c_{i})$ for each $i=1, \ldots, m$ (we denote by $\vec{\delta_{t}}$ the arrow $<\delta^{1}_{t}, \ldots, \delta^{k}_{t}>:R_{t}\to T_{1}(c_{1}) \times \cdots \times T_{1}(c_{k})$) 

\begin{center}
\begin{tikzpicture}
\node (A) at (0,0) {$c_{1}\times \cdots \times c_{k}$};
\node (B) at (2,0) {$c$};
\node (C) at (4,0) {$R_{t}$};
\node (D) at (9,0) {$T_{1}(c_{1}) \times \cdots \times T_{1}(c_{k})$};

\draw[->] (A) to node [midway, above] {$t$} (B);
\draw[dashed, ->] (B) to node [midway,above] {$\pi_{t}$} (C);
\draw[dashed, ->] (C) to node [midway,above] {$<\delta^{1}_{t}, \ldots, \delta^{k}_{t}>$} (D);
\end{tikzpicture}
\end{center}

{\flushleft
and for all terms $\vec{w}=(w_{1}:c_{1}, \ldots, c_{m} \to d_{1}, \ldots, w_{r}:c_{1}, \ldots, c_{m} \to d_{r})$ and $z:c\to d$ and $t':d_{1}, \ldots, d_{r}\to d$ (not already both in $L_{l-1}$) such that the sequent}
\[
(\top \vdash z\circ t=t' \circ \vec{w})
\] 
is provable in ${\mathbb T}_{l}$, a function symbol $R_{(t,t',\vec{w}, z)}:R_{t}\to R_{t'}$. 

We extend the translation functor $T_{1}$ to the new sorts $R_{t}$ by setting $T_{1}(R_{t})=R_{T_{1}(t)}$.

Let ${\mathbb T}_{l+1}$ be the theory obtained by expanding ${\mathbb T}_{l}$ with the following axioms: 
\[
(\top \vdash R_{(t,t',\vec{w}, z)} \circ \pi_{t}=\pi_{t'}\circ z); 
\]
\[
(\top \vdash  T_{1}(\vec{w})\circ \vec{\delta_{t}}=\vec{\delta_{t'}}\circ R_{(t,t',\vec{w}, z)})
\]

\begin{center}
\begin{tikzpicture}
\node (A) at (0,2) {$c_{1}\times \cdots \times c_{k}$};
\node (B) at (3,2) {$c$};
\node (C) at (5,2) {$R_{t}$};
\node (D) at (9,2) {$T_{1}(c_{1}) \times \cdots \times T_{1}(c_{k})$};

\node (A') at (0,0) {$d_{1}\times \cdots \times d_{r}$};
\node (B') at (3,0) {$d$};
\node (C') at (5,0) {$R_{t'}$};
\node (D') at (9,0) {$T_{1}(d_{1}) \times \cdots \times T_{1}(d_{r})$};

\draw[->] (A) to node [midway, above] {$t$} (B);
\draw[dashed,->] (B) to node [midway,above] {$\pi_{t}$} (C);
\draw[dashed, ->] (C) to node [midway,above] {$\vec{\delta_{t}}$} (D);

\draw[->] (A) to node [midway, right] {$<w_{1}, \ldots, w_{r}>$} (A'); 

\draw[->] (B) to node [midway, right] {$z$} (B'); 

\draw[dashed, ->] (C) to node [midway, right] {$R_{(t,t',\vec{w}, z)}$} (C');
 
\draw[->] (D) to node [midway, right] {$T_{1}(<w_{1}, \ldots, w_{r}>)$} (D');
 
\draw[->] (A') to node [midway, above] {$t$} (B');
\draw[dashed, ->] (B') to node [midway,above] {$\pi_{t}$} (C');
\draw[dashed, ->] (C') to node [midway,above] {$\vec{\delta_{t'}}$} (D');
\end{tikzpicture}
\end{center}
{\flushleft
plus the functoriality axioms}
\[
(\top \vdash R_{(t,t',\vec{w}, z)}=id)
\]
if $t=t'$, $\vec{w}=id$ and $z=id$;
\[
(\top \vdash  R_{(t',t'',\vec{w'}, z')} \circ R_{(t,t',\vec{w}, z)}=R_{(t,t'',\vec{w'} \circ \vec{w}, z'\circ z)}).
\]

Concerning distinguished triangles of the algebraic syntactic category ${\cal C}_{l+1}$ of the theory ${\mathbb T}_{l+1}$, we stipulate that they are those of the syntactic category ${\cal C}_{l}$ plus all the trivial ones, the ones isomorphic in ${\cal C}_{l}$ to those of the form $(t, \pi_{t}, \delta_{t})$, $(\pi_{t}, \delta_{t}, T_{1}(t))$, $(\delta_{t}, T_{1}(t), \pi_{T_{1}(t)})$ for all terms $t$ over the signature $L_{l}$ (so that the triangle functoriality axiom for triangulated categories is satisfied), and the ones isomorphic in ${\cal C}_{l}$ to those of the form $(R_{(u,v\circ u,id, v)}, R_{(v\circ u, v, u, id)}, T_{1}(\pi_{u})\circ \delta_{v})$ (for any composable terms $u$ and $v$ over $L_{l}$) and all its translations (so that the octahedral axiom is satisfied).  
  
We set $\cal C$ equal to the union of all the categories ${\cal C}_{l}$; we define the set of its distinguished triangles to be the union of the sets of distinguished triangles of the categories ${\cal C}_{l}$. The translation functor for the category $\cal C$ is simply given by the extension of the translation functors for the subcategories ${\cal C}_{l}$. As we shall see below, unlike the categories ${\cal C}_{l}$, the category $\cal C$ is triangulated. 

We denote by $\mathbb T$ the union of all the theories ${\mathbb T}_{l}$. The category $\cal C$ is the syntactic category of the theory $\mathbb T$. Indeed, the arrows in the latter are $\mathbb T$-provable equivalence classes of terms over the union $L$ of all the signatures $L_{l}$, and it is immediate to see that, given two terms $t$ and $t'$ lying in $L_{l}$, the sequent $(\top \vdash t=t')$ is provable in $\mathbb T$ if and only if it is provable in some ${\mathbb T}_{k}$, if and only if it is provable in ${\mathbb T}_{l}$; so all the arrows in $\cal C$ are in fact arrows in a category ${\cal C}_{l}$ for some $l$. 

\begin{theorem}
The category $\cal C$, equipped with the translation functor and the set of distinguished triangles specified above, is a triangulated category in which the filling conditions in the axioms of triangulated categories are canonically (but not necessarily uniquely) satisfied. 
\end{theorem}

\begin{proof}
The inductive construction of $\cal C$ makes it possible for any morphism in a category ${\cal C}_{l}$ to have a canonically defined cone on it lying in the subcategory ${\cal C}_{l+1}$ (namely, $\pi_{t}$ if $t$ lies in ${\cal C}_{l}$ for $l$ minimal with this property). The definition of the distinguished triangles in $\cal C$ makes all the other axioms of triangulated categories canonically satisfied. Still, uniqueness does not hold, due to the requirement, in the definition of triangulated category, that any triangle which is isomorphic to a distinguished triangle should be distinguished (which has forced us to take all isomorphic copies of our canonical triangles in the definition of the distinguished triangles of $\cal C$). In fact, there might be more than one isomorphism between two (isomorphic) mapping cones. For instance, for a given morphism of schemes $u:X\to Y$, one can have isomorphic ways to complete it to a distinguished triangle: the canonical one $\pi_{u}:Y \to R_{u}$, and any other morphism $\pi_{u'\circ g}:Y \to R_{u'}$ where $u'$ and $g$ are respectively a morphism $X' \to Y'$ and an isomorphism $Y\to Y'$ such that $u'\circ f=g\circ u$ for some isomorphism $f$. The isomorphism between these two cones is given by $R_{(u, u', f, g)}$ and in general, for fixed $u, u'$ and $\pi_{u'}\circ g$, there can clearly be many such isomorphisms between them.  
\end{proof}

\begin{remarks}
\begin{enumerate}[(i)]
\item In the theory ${\mathbb T}'$ obtained from $\mathbb T$ by adding the axioms $(\top \vdash g\circ f=0)$ for each consecutive arrows $f$ and $g$ in a distinguished triangle of $\cal C$, every arrow in the syntactic category of ${\mathbb T}'$ has an equationally definable image. This is relevant in relation to our goal of showing that the known homology theories are homogeneous since under the hypotheses of Theorem \ref{thm_axiomatizationhomogeneous}, the sequent
\[
( R_{\vec{w}_{\ast}(S)}(\vec{y}) \wedge  \vec{w}_{\ast}(U)(\vec{y})  \vdash_{\vec{y}} (\exists \vec{x}) (\vec{y}=\vec{w}(\vec{x})))
\]
is satisfied by every homology theory as a consequence of exactness. This reduces the problem of proving homogeneity to that of checking the additional requirement $R_{S}(\vec{x}) \wedge U(\vec{x})$.

\item The dual of the triangulated category $\cal C$ maps into all the (triangulated) derived categories of sheaves on a given site. The relationship between our triangulated category and the formalism of derived categories is an interesting one. In the derived formalism the objects are complexes, the translation functor as well as the cone construction operate on them rather than on single (co)homology groups or arrows between them. Our construction is \ac punctual' and minimalist rather than \ac global' as the derived formalism, in that it exploits the fact that one can define derived functors starting from just an arrow rather than from a whole complex. Our construction extracts the essential structural features of the construction of (co)homology theories of schemes; indeed, in any geometric situation one can define the cone of a morphism $f:X\to Y$ in a completely canonical way.      

\item The theory $\mathbb T$ should play the role of the primitive theory from which we start in order to construct our theory of presheaf type as described in section \ref{sec:motivictheories}. Indeed, it extends in a natural way the empty theory over the language associated to Nori's quiver and, as remarked in point (i), reduces the strength of the homogeneity condition in relation to exactness. 

\item The cohomology of constant sheaves as well as of all the sheaves which can be obtained by applying geometric operations to them such as direct image is present in this framework.
\end{enumerate}
\end{remarks} 

\vspace{0.5cm}

\textbf{Acknowledgements:} I am very grateful to Luca Barbieri-Viale and Laurent Lafforgue for useful discussions on the subject matter of this paper, in particular for pointing out that the notion of triangulated category might be relevant for making all the images of arrows coming from Nori's diagram equationally definable.

\vspace{1cm}

\textsc{Olivia Caramello}

{\small \textsc{UFR de Math\'ematiques, Universit\'e de Paris VII, B\^atiment Sophie Germain, 5 rue Thomas Mann, 75205 Paris CEDEX 13, France}\\
\emph{E-mail address:} \texttt{olivia@oliviacaramello.com}}

\end{document}